\documentclass[letterpaper,11 pt,reqno]{amsart}

\usepackage[margin=1.2in, bottom=2in]{geometry}

\usepackage{graphicx,amsmath,amssymb,amsthm,paralist,color,tikz-cd,comment}
\usepackage{mathrsfs}
\usepackage{mathtools}
\usepackage{cite}
\usepackage{setspace}
\usepackage{enumerate}
\usepackage{graphicx}

\usepackage{enumitem}

\usepackage{tikz}

\numberwithin{equation}{section}

\usepackage[color=blue!30]{todonotes}
\reversemarginpar
\usepackage{hyperref}
\hypersetup{
    colorlinks   = true,
     citecolor    = black,
    linkcolor    = blue
}

\newtheorem{thmintro}{Theorem}

\newtheorem{corintro}[thmintro]{Corollary}

\newtheorem{theorem}{Theorem}[section]
\newtheorem{lemma}[theorem]{Lemma}
\newtheorem{corollary}[theorem]{Corollary}
\newtheorem{proposition}[theorem]{Proposition}

\theoremstyle{definition}

\newtheorem{fact}[theorem]{Fact}
\newtheorem*{remark*}{Remark}

\newtheorem{remark}[theorem]{Remark}

\newcommand{\field}[1]{\mathbb{#1}}
\newcommand{\R}{\field{R}}

\renewcommand{\phi}{\varphi}

\newcommand{\eps}{\varepsilon}
\newcommand{\RomanNumeralCaps}[1]
    {\MakeUppercase{\romannumeral #1}}
\newcommand{\II}{\text{\RomanNumeralCaps{2}}}

\renewcommand{\|}{\,\Vert\,}

\usepackage{fancyhdr}
\title{Riemannian Anosov extension and applications}
\author{Dong Chen, Alena Erchenko, and Andrey Gogolev}
\thanks{
The second author was partially supported by NSF grant DMS-1547145. The third author was partially supported by NSF grants DMS-1823150.}

\newcommand{\Addresses}{{
\bigskip
  \footnotesize

  \textsc{Department of Mathematics, The Ohio State University, Columbus, OH}
	\par\nopagebreak
  \textit{E-mail address:} \texttt{chen.8022@osu.edu}
  \bigskip
  \footnotesize

  \textsc{Department of Mathematics, The University of Chicago, Chicago, IL}
	\par\nopagebreak
  \textit{E-mail address:} \texttt{aerchenko@uchicago.edu}
	
	\bigskip
  \footnotesize

  \textsc{Department of Mathematics, The Ohio State University, Columbus, OH}
	\par\nopagebreak
  \textit{E-mail address:} \texttt{gogolyev.1@osu.edu}

}}

\begin{document}
\baselineskip=14pt

\begin{abstract}
  \begin{sloppypar}
  Let $\Sigma$ be a Riemannian manifold with strictly convex spherical boundary. Assuming absence of conjugate points and that the trapped set is hyperbolic, we show that $\Sigma$ can be isometrically embedded into a closed Riemannian manifold with Anosov geodesic flow. We use this embedding to provide a direct link between the classical Livshits theorem for Anosov flows and the Livshits theorem for the X-ray transform which appears in the boundary rigidity program. Also, we give an application for lens rigidity in a conformal class. 
    \end{sloppypar}
\end{abstract}
\maketitle

\section{Introduction}

A closed Riemannian manifold $(M,g)$ is called {\it Anosov} if the corresponding geodesic flow on the unit tangent bundle $T^1M$ is an Anosov flow. For example, all closed manifolds with strictly negative curvature are Anosov. Special examples of manifolds which are not negatively curved, but carry Anosov geodesic flows are known. The first one is probably due to Eberlein~\cite{E73} who performed a careful local deformation of a hyperbolic manifold to create a small disk of zero curvature. Due to the $C^1$ stability of the Anosov property, Eberlein's example can be perturbed further to create some positive curvature while keeping the Anosov property. Further examples were constructed by Gulliver \cite{G75}, using radially symmetric caps of positive curvature, and by Donnay-Pugh~\cite{DP03} who constructed Anosov surfaces embedded in $\mathbb R^3$.  It is shown in a recent paper \cite{DSW21} that for a geodesic billiard system whose trapped set is hyperbolic and non-grazing, it is possible to produce a smooth model of Axiom A ﬂow for the discontinuous ﬂow deﬁned by the non-grazing billiard trajectories. 

Our main result shows that one can embed certain Riemannian manifolds $(\Sigma, g)$ with boundary and hyperbolic trapped sets isometrically into an Anosov manifold (Recall that the trapped set is the set of geodesics that are defined for all time, and a boundary is called strictly convex if its second fundamental form is positive definite everywhere).

\begin{thmintro}[Theorem ~\ref{thm_anosov_ext}]\label{thm_anosov_intro}
 Let $(\Sigma, g)$ be a compact smooth Riemannian manifold with boundary. Assume that each component of the boundary is a strictly convex set diffeomorphic to a sphere. Also, assume that $(\Sigma, g)$ has no conjugate points and the trapped set for the geodesic flow is hyperbolic. Then, there exists a codimension 0 isometric embedding $(\Sigma,g)\subset (\Sigma^{ext}, g^{ext})$ such that $(\Sigma^{ext}, g^{ext})$ is a closed Anosov manifold.
\end{thmintro}
\begin{remark*}We do not require $\Sigma$ to be connected. If we do not insist on the embedding being codimensional 0 then it is not hard to apply Nash's embedding theorem to isometrically embed $(\Sigma, g)$ into a high dimensional Euclidean space and then into a horosphere in a manifold of constant negative curvature (We owe this remark to Keith Burns).
\end{remark*}
To the best of our knowledge, the above theorem is the first general result on existence of Anosov extensions. We note that all assumptions except for convexity and diffeomorphism type of the boundary are necessary assumptions to admit an Anosov extension. One fact which immediately follows from  Theorem \ref{thm_anosov_intro} is that for any point in any Riemannian manifold, one can isometrically embed any sufficiently small neighborhood of the given point into a closed Anosov manifold. 

Theorem \ref{thm_anosov_intro} allows one to transfer some results from the setting of closed Riemannian manifolds to the setting of compact Riemannian manifolds with boundary. We proceed with a description of such applications.

Denote by $\partial_-$ (respectively, $\partial_+$) the unit inward (respectively, outward) vectors based on $\partial \Sigma$ (precise definition are given in Section \ref{prel: hyp_trapped_sets}). The {\it lens data} consists of two parts: the length map $l_g: \partial_-\to [0,\infty]$ measuring the time at which $\gamma_v$ hits $\partial\Sigma$ again for all $v\in\partial_-$, and the scattering map $s_g: \partial_-\setminus\Gamma_-\to \partial_+$ associating $v\in \partial_-\setminus\Gamma_-$ with its exiting vector $s_g(v)$. Here $\Gamma_-:=l_g^{-1}(\infty)$. We say that two metrics $g$ and $g'$ on $\Sigma$ are \textit{lens equivalent} if $l_g=l_{g'}$ and $s_g=s_{g'}$. For any metric $g$ on $\Sigma$, denote by $g_U$ the lifted metric on the universal cover $\tilde{\Sigma}$. Two metrics $g$ and $g'$ on $\Sigma$ are called \textit{marked lens equivalent} if the lens data of $g_U$ and $g'_U$ coincide. The \textit{lens rigidity (resp. marked lens rigidity)} problem asks whether lens equivalent (resp. marked lens equivalent) metrics are isometric via a diffeomorphism fixing $\partial \Sigma$. 

Together with an argument of Katok~\cite{K88}, we confirm the following extension of Mukhometov-Romanov result \cite{MR78} in the case when hyperbolic trapped sets are allowed.
\begin{corintro}[Marked lens rigidity in a conformal class]
\label{cor_conf_class}
Let $\rho\colon \Sigma\to\R_+$ be a smooth function such that the metrics $(\Sigma, g)$ and $(\Sigma, \rho^2 g)$ both satisfy the assumptions in Theorem~\ref{thm_anosov_intro}. Assume that $g$ and $\rho^2 g$ are marked lens equivalent. Then, $\rho=1$.
\end{corintro}

\begin{remark*}
Corollary \ref{cor_conf_class} 
is related to  the boundary rigidity problem, which asks whether one can reconstruct the Riemannian metric $g$ in the interior from knowing the distance $d_g: \partial \Sigma\times \partial \Sigma\to\mathbb{R}$ between points on the boundary. Michel \cite{M81} conjectured that all simple manifolds are boundary rigid, and the surface case was proved by Pestov-Uhlmann \cite{PU05}. Partial results in higher dimensions can be found in \cite{SU09}, \cite{V09}, \cite{BI10}, \cite{BI13}, \cite{SUV17}, etc. When trapped sets are allowed, the marked lens rigidity is equivalent to marked boundary rigidity, and certain local rigidity results were recently established in \cite{G17}, \cite{GM18}, \cite{L19}, and \cite{L20} in the case when trapped sets are hyperbolic. 
\end{remark*}

Another application is a smooth Livshits theorem for domains with sharp control of regularity of the solution. 

\begin{corintro}[Livshits Theorem for domains]
\label{thm_Livshits}
Let $(\Sigma, g)$ be as in Theorem~\ref{thm_anosov_intro} and let a $C^r$-smooth ($r>0$) function  $\beta\colon T^1\Sigma\to\R$ be  such that its $C^r$-jet vanishes on the boundary $\partial T^1\Sigma$. Assume that for all $v\in\partial_-\backslash \Gamma_-,$
$$
\int_{0}^{l_g(v)}\beta(\gamma_v(t))dt=0.
$$
Then, there exists $u\in C^{r_-}(T^1\Sigma)$ such that $Xu=\beta$ and $u|_{\partial(T^1\Sigma)}=0$, where $X$ is the geodesic spray.
\end{corintro}

Here $r_-=r$ if $r$ is not an integer. If $r$ is an integer then $r_-=r-1+\textup{Lip}$.  Corollary \ref{thm_Livshits} was also proved in \cite[Proposition 5.5]{G17}, and the proof there applies to $u\in H^s(T^1\Sigma)$ with $s>0$. Our proof is more geometric and covers the H\"{o}lder regularity. 

\begin{remark*}
The reason why Livshits theorem is restricted to functions which are flat on the boundary is that, otherwise, the standard bootstrap argument for solution of the cohomological equation~\cite{LMM86} does not work. However, notice that our condition is not a restriction for the potential application to the deformation lens rigidity (as in \cite{G17}) due to a result of Lassas-Sharafutdinov-Uhlmann who recover the jet of the metric from local lens data~\cite{LSU03}. 
\end{remark*}

\begin{remark*} All our results have low regularity versions in the case when $(\Sigma, g)$ has finite regularity which exceeds $C^{3+\alpha}$ for some positive $\alpha>0$.
\end{remark*}

\begin{remark*} 
The basic example to which Theorem A applies is, of course, when $\Sigma$ is a strictly convex ball equipped with a simple metric $g$. In this case the trapped set must be empty since $(\Sigma, g)$ is assumed to have no conjugate points. When $\dim \Sigma=2$ it is easy to make examples which have arbitrary genus, if the genus $>0$ then the trapped set is non-empty. It was pointed out to us by one of the referees that examples which satisfy all assumptions of Theorem A and have non-empty trapped set might not exist in dimensions $\ge 3$.  While we do not know how to prove that this is the case, we agree that the existence of such example, indeed, seems to be unlikely. We would like to point out that interesting higher dimensional examples with non-empty trapped set exist. While, formally speaking, these examples are not covered by Theorem A, existence of Anosov extension for such examples still holds with some adjustments to the proof of Theorem A.

Let $\gamma$ be a closed geodesic in a negatively curved manifold $M$ which does not have self-intersections. Then a small neighborhood $\Sigma$ of the ``core’’ $\gamma$ in $M$ satisfies all the assumptions of Theorem~A except that $\partial\Sigma\simeq S^1\times S^{n-2}$. Note that $\gamma$ constitutes a non-trivial hyperbolic trapped set for $\Sigma$. (Alternatively one can obtain such example by explicitly specifying a negatively curved metric on $\Sigma=S^1\times\mathbb D^n$.)  We note that $\Sigma$ already satisfies the conclusion of Theorem~A since it is isometrically embedded in $M$. However, one can deform the metric, for example by creating islands of positive curvature away from $\partial\Sigma$ and $\gamma$, such that existence of Anosov embedding becomes in no way obvious. For this class of examples the proof remains exactly the same up to Section~\ref{sec: anosov_ext}, where we take advantage of spherical boundary to glue in out extended domains into a hyperbolic manifold with a large injectivity radius. This argument, with some work, can be adjusted to accommodate the above example. Specifically, the large hyperbolic manifold has to be replaces with a hyperbolic manifold which contains a ``large geometric tube’’ with core $\gamma$. Existence of hyperbolic manifolds which contain such ``large geometric tubes’’ was established by Farrell and Jones~\cite[Corollary 3.3]{FJ93} who construct them via a carefully chosen finite cover.
\end{remark*}

\subsection{Outline of the proof of Theorem~\ref{thm_anosov_intro}}\label{outline of the proof} We construct the extension by hand. Firstly, for each boundary component of the given manifold,  we find a metric on a collar that smoothly connects the metric on this boundary to a constant curvature metric. Afterwards, we throw away from a compact manifold of a constant sectional curvature (which has the same dimension as the given manifold) finitely many balls (as many as the number of boundary components in the original manifold) that are sufficiently far away (see Lemma~\ref{lemma: large injectivity radius} and the paragraph before it). Finally, in the resulting manifold with boundary and constant negative curvature, we glue in the given manifold with attached collars. The metric on the collar is constructed in several steps. First, we extend the given metric on the neighborhood of the boundary to the negatively curved metric (Section~\ref{construction 1}). Then, we connect the resulting metric to a rotationally invariant metric in the cylindrical coordinates (Section~\ref{construction 2}). Finally, we extend the result of the previous extension to a metric of constant curvature (Appendix~\ref{app: combine f with construction 2}). The original metric in the collar is $C^{1,1}$ but we smooth it afterwards (Section~\ref{sec: metric_ext}).

To guarantee that the constructed compact Riemannian manifold has Anosov geodesic flow we use the criterion by Eberlein (see Theorem~\ref{EberleinAnosov}). In particular, we first show that the constructed metric does not have conjugate points. Then, we prove that all nonzero perpendicular Jacobi fields are unbounded. Instead of working directly with the Jacobi field, we estimate the growth rates $\mu$ of the logarithm of the square of the norm of nonzero perpendicular Jacobi fields (see ~\eqref{def: mu}) using the comparison Lemma~\ref{lem: comp_thm}. In particular, the absence of conjugate points means that there is no time interval so that $\mu$ tends to infinity as we approach each end of the interval (Proposition~\ref{prop: noconjpts}). By Lemma~\ref{lem: comp_thm} and Remark~\ref{rem: behavior of Riccati solution}, we will need to control what are the values of $\mu$ as the geodesic enters various regions (the given manifold with boundary and various extension pieces that we construct to obtain the compact manifold with Anosov geodesic flow) so that we have a control from below while it is in the specific region (see Figure~\ref{fig: graph_of_mu}). To show that all nonzero perpendicular Jacobi fields are unbounded, it is enough to show that the integral of $\mu$ over a time ray is unbounded.

\subsection{Organization} This paper is organized as follows. In Section \ref{sec: prelim} we set up notation and collect a number of preliminaries from geometry and dynamics. In Section~\ref{sec: proof_corollaries} we prove  Corollaries ~\ref{cor_conf_class} and~\ref{thm_Livshits} using Theorem \ref{thm_anosov_intro}. The estimates for Jacobi field within a slightly larger domain containing $\Sigma$ are carried out in Section \ref{sec: mu_estimate_domain}. The estimates on curvature for certain extension are presented in Sections \ref{construction 1}-\ref{sec: metric_ext}. In Section \ref{sec: anosov_ext} we construct an explicit extension of the metric and prove Theorem~\ref{thm_anosov_intro}.

\subsection{Acknowledgment}  The authors would like to express their gratitude to the referees for valuable suggestions on the improvement of the paper.

 \section{Preliminaries}\label{sec: prelim}

\subsection{Geometry of the tangent bundle}\label{section: geometry tangent bundle}

In this section, we formulate some general facts about the tangent bundle. One can find more details in \cite{E73} and \cite{EO80}.

Let $(M,g)$ be a $C^{2+\alpha}$, $\alpha>0$, $n$-dimensional compact Riemannian manifold with or without a boundary. Denote by $T^1 M$ the unit tangent bundle of $M$.  For any $v\in T^1M$, let $\gamma_{v}$ be the unit speed geodesic in $(M,g)$ such that $\gamma'_{v}(0)=v$. The geodesic flow $\phi_t: T^1M\to T^1M$ is defined by setting $\phi_t(v)=\gamma'_{v}(t)$. A vector field $J(t)$ along $\gamma_v$ is a \textit{Jacobi field} if $J(t)$ satisfies the Jacobi equation
\begin{equation}\label{Jacobi_equation}
J''(t)+R({J(t), \gamma_v'(t)})\gamma_v'(t)=0,
\end{equation}
where $R$ is the Riemann curvature tensor and $'$ corresponds to the covariant differentiation along $\gamma_v$. A Jacobi field is uniquely determined by the values $J(0)$ and $J'(0)$. 

Denote by $\pi: TM\to M$ the canonical projection. For any $\xi\in TTM$, let $c(t)$, for $t\in(-\eps,\eps)$, be a curve on $TM$ with $c'(0)=\xi$. Define the {\it connection map $K: TTM\to TM$} by $K\xi=\nabla_{\pi\circ c}c(0)$. It is well-defined since $\nabla_{\pi\circ c}c(0)$ is independent of the choice of $c$. The map $d\pi\oplus K: TTM\to TM\oplus TM$ is a linear isomorphism. The kernel of $d\pi: TTM\to TM$, denoted by $H$, is called {\it the  horizontal subbundle}, while the kernel $V$ of the connection map $K$ is called {\it the vertical subbundle}. The {\it Sasaki metric} on $TTM$ is defined via
$$\langle \xi, \eta\rangle:=g_{\pi v} (d\pi \xi, d\pi\eta)+g_{\pi v} (K\xi, K\eta)$$
for $\xi,\eta\in T_vTM$. We denote by $|\xi|:=\sqrt{\langle \xi, \xi\rangle}$ the Sasaki norm of $\xi\in TTM$.

\begin{fact}\label{identification_Jacobi}
Now vectors in the tangent space $T_vT^1M$ can be identified with Jacobi fields along $\gamma_v$ in the following way: for any $\xi\in T_v(T^1M)$, we define $J_{\xi}$ to be the unique Jacobi field along $\gamma_v$ with $J_{\xi}(0)=d\pi\xi$ and $J'_{\xi}(0)=K\xi$. 
\end{fact}

The above identification is invariant under the geodesic flow, namely, 
\begin{equation*}
J_{D\phi_t(\xi)}(0)=J_{\xi}(t)\qquad \text{and} \qquad J'_{D\phi_t(\xi)}(0)=J'_{\xi}(t).
\end{equation*} 
In particular, if we fix $\xi\in T_vT^1M$ then $g_{\pi v} (J_{\xi}(t), \gamma'_v(t))$ is independent of $t$. Thus, for any $\xi\in T_vT^1M$, $J_{\xi}$ is perpendicular to $\gamma_v$ if and only if $\langle\xi, X\rangle=0$ where $X$ is the vector field on $TM$ generating the geodesic flow $\phi_t$ on $(M,g)$. We denote the space of Jacobi fields perpendicular to a geodesic $\gamma$ by $\mathcal J(\gamma)$.  

Note that the Sasaki norm of $(d\phi_t)\xi$ is given by
\begin{equation}\label{eq: sasaki_jacobi}
|(d\phi_t)\xi|^2=\|J_{\xi}\|^2(t)+\|J'_{\xi}\|^2(t).
\end{equation}

\subsection{Hyperbolicity} Let $\phi_t\colon M\to M$ be a smooth flow on a Riemannian manifold and let $X$ be its generating vector field.
Recall that an invariant set $\Lambda$ is \textit{$\lambda$-hyperbolic} (where $\lambda>0$) if there exist $C>0$ and a continuous flow-invariant splitting
\begin{equation*}
T_\Lambda M = \mathbb R X\oplus E^u\oplus E^s
\end{equation*}
such that for all $y\in\Lambda$,
\begin{align}
&\|d\phi_t(y)w\|\leq Ce^{-\lambda t}\|w\|, \qquad\forall t>0, \forall w\in E^s(y)\qquad \text{and}\label{su_exp}\\
&\|d\phi_t(y)w\|\leq Ce^{\lambda t}\|w\|, \qquad\forall t<0, \forall w\in E^u(y),\nonumber
\end{align}
where $\|\cdot\|$ is the norm on $T_yM$ induced by the Riemannian metric.
Distributions $E^s$ and $E^u$ are called \textit{stable and unstable subbundles on $\Lambda$}.

If $\Lambda=M$ then $\phi_t$ is called an {\it Anosov flow.} For Anosov flows the classical Livshits Theorem is stated as follows.
\begin{theorem}[\hspace{1sp}\cite{L71, LMM86}] 
\label{Liv}
Let $\varphi_t\colon M\to M$ be a transitive Anosov flow and let $\beta\colon M\to\R$ be a $C^r$ function such that
$$
\int_\gamma\beta(\gamma(t))dt=0
$$
for every periodic orbit $\gamma$. Then there exists $u\in C^{r_-}(M)$ such that $Xu=\beta$, where $X$ is the generator for the geodesic flow.
\end{theorem}

Recall that $r_-=r$ if $r$ is not an integer and $r_-=r-1+\textup{Lip}$ when $r$ is an integer. We will use the following criterion, due to Eberlein, for establishing the Anosov property of geodesic flows. Another proof of this criterion was given Ruggiero \cite{R07} following an idea of M\'{a}n\~{e}. 

\begin{theorem}[\hspace{1sp}\cite{E73}, see also~\cite{R07}]
\label{EberleinAnosov}
Let $\phi_t$ be a geodesic flow on a closed Riemannian
 manifold without conjugate points. Then $\phi_t$ is Anosov if and only if all nonzero perpendicular Jacobi fields are unbounded. \end{theorem}

When $\Lambda\neq M$, the following result lets us  extend the hyperbolic structure to a neighborhood of $\Lambda$.
\begin{lemma}[\hspace{1sp}\cite{HPPS70}]\label{lem: loc_hyp}
Let $\Lambda$ be a $\lambda$-hyperbolic set. Then for any $\eps\in (0,\lambda)$,
 there exists an open neighborhood $\mathcal V_\eps$ of $\Lambda$ and extensions $E^s$ and $E^u$ of the stable and unstable subbundles to  $\mathcal V_\eps$ with the following properties:
\begin{enumerate}
\item\label{loc_inv} Local invariance: if an orbit segment $[y, \phi^t(y)]\subset \mathcal V_\eps$ then, $d\phi_t(y)E^s(y)=E^s(\phi_t(y))$ and $d\phi_t(y)E^u(y)=E^u(\phi_t(y))$
\item\label{rate_close} Hyperbolicity: if an orbit segment $[y, \phi^t(y)]\subset \mathcal V_\eps$ then
\begin{align*}
&\|d\phi_t(y)w\|\leq Ce^{-(\lambda-\eps)t}\|w\|, \qquad \forall w\in E^s(y)\qquad\text{and}\\
&\|d\phi_t(y)w\|\geq \frac 1Ce^{(\lambda-\eps)t}\|w\|, \qquad  \forall w\in E^u(y).
\end{align*}
\end{enumerate}
\end{lemma}

\begin{remark} 
The reference~\cite{HPPS70} does not contain an explicit statement about the hyperbolic rate being close to $\lambda$ (item~\eqref{rate_close} in Lemma~\ref{lem: loc_hyp}). However, this rate, indeed can be chosen as close to $\lambda$ as desired by choosing a sufficiently small neighborhood of $\Lambda$. This follows from the fact that the expansion and contraction rates depend continuously on the point. In the case when $M$ is 3-dimensional such extensions of bundles $E^s$ and $E^u$ can be chosen so that they integrate to locally invariant continuous foliations. In a higher dimension this seems to be unknown. However, for our purposes we will merely need locally invariant bundles which do not necessarily integrate to foliations.
\end{remark}

\subsection{The hyperbolic trapped set}\label{prel: hyp_trapped_sets}
Let $(\Sigma,g)$ be a smooth $n$-dimensional compact Riemannian manifold with boundary. Denote by $\pi: T^1{\Sigma}\to \Sigma$ the canonical projection and 
 $\left(T^1\Sigma\right)^\circ$ the interior of $T^1\Sigma$. Let $\partial_{-}$ and $\partial_{+}$ be the incoming and outgoing, respectively, subsets of the boundary of $T^1\Sigma$ defined by
$$\partial_{\pm} = \left\{v\in T^1\Sigma\,|\, \pi(v)\in\partial\Sigma, \, \pm g(v,\nu)>0\right\},$$
where $\nu$ is the unit normal vector field to $\partial\Sigma$ pointing outwards. For any $v\in \partial_-$, the geodesic $\gamma_v$ starting at $v$ either has an infinite length or exits $\Sigma$ at a boundary point with tangent vector
 in $\partial_+$. We denote by $l_g(v)\in [0,\infty]$ the length of $\gamma_{v}$ in $\Sigma$. Let 
$$\Gamma_{-}:=\{v\in \partial_{-}: l_g(v)=\infty\}.$$
For each $v\in\partial_-\setminus \Gamma_-$ denote the exit point by $s_g(v)\in\partial_+$. Similarly we define the set $\Gamma_+\subset\partial_+$ which is trapped in $\Sigma$ in backwards time. Then the {\it trapped set} in the interior of $\Sigma$ is defined via
$$
\Lambda=\Gamma_-\cap \Gamma_+.
$$
If $\Sigma$ has strictly convex boundary, $\Lambda$ is the set of $v\in (T^1\Sigma)^\circ$ such that the $\phi_t$-orbit of $v$ does not intersect the boundary. Throughout this paper we will always assume that the trapped set is hyperbolic. The stable and unstable bundles of $\Lambda$ are denoted $E^s$ and $E^u$, respectively. It is clear from 
the discussion in Section \ref{section: geometry tangent bundle} and \eqref{su_exp} that both $E^s$ and $E^u$ are perpendicular to the generating vector field~$X$. 

\begin{remark}\label{rem: U+-}
When $v\in\Lambda$, the isomorphism $d\pi\oplus K: TTM\to TM\oplus TM$ maps the invariant subbundles $E^{\sigma}(v)$ to the graph of stable/unstable Riccati
 tensors $U^{\sigma}_v$ on $v^{\bot} := \{w\in T_{\pi(v)}M| g_{\pi(v)}(w,v)=0\}$, $\sigma=s,u$. See \cite{E73} for more details.
\end{remark}

Now we apply Lemma~\ref{lem: loc_hyp} to the trapped set with $\eps=\frac{\lambda}{2}$. If $v\notin\Lambda$ then the invariant subbundles along the orbit through $v$ only exist for a finite time and, hence, 
they do not have to be perpendicular to $X$. Nevertheless, we can still 
obtain perpendicular invariant bundles by taking the orthogonal component (which only results in a slightly different constant $C$ in Lemma~\ref{lem: loc_hyp}). 

More specifically, we define the following linear subspaces of the space of Jacobi fields along a geodesic $\gamma_v$:
\begin{equation*} 
\mathcal J^{\sigma}(v) = \{J_\xi| \xi\in E^{\sigma}(v)\} \qquad\text{and} \qquad \mathcal J_{\bot}^{\sigma}(v) = \{J^{\bot}_\xi|\xi\in E^{\sigma}(v)\},
\end{equation*}
where $J_{\xi}=J^{\bot}_{\xi}+J^{\parallel}_{\xi}$ with $J^{\bot}_{\xi}$ being a perpendicular Jacobi vector field and $J^{\parallel}_{\xi}$ being a tangential Jacobi vector field, i.e., $J^{\parallel}_{\xi}(t)=(\alpha t+\beta)\gamma_v'(t)$ for some $\alpha,\beta\in\mathbb R$.  For any $\sigma\in\{s, u\}$, let $E^{\sigma}_{\bot}(v) := \{\xi\in T_vT^1M| J_{\xi}\in\mathcal J^{\sigma}_{\bot}(v)\}$.

Now we have the following variant of Lemma~\ref{lem: loc_hyp} near the hyperbolic trapped set $\Lambda$ of the geodesic flow.

\begin{lemma}\label{lem: perp_component}
There exists a neighborhood $\mathcal U$ of $\Lambda$ such that $\mathcal U\subset \mathcal V_{\frac{\lambda}{2}}$ and for $\sigma\in\{s, u\}$,
\begin{enumerate}
\item\label{cont_inv_bundle} $E^{\sigma}_{\bot}$ are continuous subbundles in $\mathcal U$;
\item $T_v(T^1\Sigma) = \mathbb RX(v)\oplus E^u_{\bot}(v)\oplus E^s_{\bot}(v)$ for all $v\in\mathcal U$;
\item For any $v\in\mathcal U$,  we denote by $(t_{-}(v), t_{+}(v))$ the maximal time interval on which $\phi_t(v)\in \mathcal U$. Then we have $d\phi_t(v)E_{\bot}^{\sigma}(v)=E^{\sigma}_{\bot}(\phi_t(v))$ for all $t\in(t_-(v),t_+(v))$;
\item\label{rate_for_perp} There exists $C'>0$ such that for all $v\in\mathcal U$, 
\begin{align*}
&|d\phi_t(v)\xi|\leq C'e^{-\frac{\lambda}{2}t}|\xi|, \qquad \forall t\in(0,t_+(v)), \forall \xi\in E^s_{\perp}(v)\qquad\text{and}\\
&|d\phi_t(y)\xi|\leq C'e^{\frac{\lambda}{2}t}|\xi|, \qquad \forall t\in(t_-(v),0), \forall \xi\in E^u_{\perp}(v);
\end{align*}
\item\label{linear_isom} Let $v\in \mathcal U$ and $v^{\bot} = \{w\in T_{\pi(v)}M| g_{\pi(v)}(w,v)=0\}$. For each $w\in v^{\bot}$, there exists a unique vector $\xi^{\sigma}_{w}\in E^{\sigma}_{\bot}(v)$ such that $d\pi\xi^{\sigma}_w=w$, and the map $X^{\sigma}_v: v^{\bot}\to E^{\sigma}_{\bot}(v), w\mapsto \xi^{\sigma}_w$ is a linear isomorphism.

\item\label{norm_bound_in_U} The map $U^{\sigma}_v:=K\circ X^{\sigma}_v$ is a linear endomorphism on $v^{\bot}$ and there exists $L>0$ depending only on $\Lambda$ such that $|U^{\sigma}_v|\leq L$ for all $v\in \mathcal U$.

\end{enumerate}
\end{lemma}

\begin{proof}
By Lemma~\ref{lem: loc_hyp}, $E^{\sigma}$ are continuous and invariant under the flow $\{\phi_t\}$ in $\mathcal V_{\frac{\lambda}{2}}$, so we obtain the first three items of the lemma because the splitting into perpendicular and tangential Jacobi vector fields is invariant under the flow.

Since $E^{\sigma}_{\bot}(v)=E^{\sigma}(v)$ for all $v\in\Lambda$, there exists a neighborhood $\mathcal U$ of $\Lambda$ such that for any $v\in\mathcal U$ for any $\xi\in E^u(v)\cup E^s(v)\setminus\{0\}$ we have $\frac{|<\xi,X>}{\|\xi\|}|\leq \frac{1}{10}$. Thus, using Lemma~\ref{lem: loc_hyp} \eqref{rate_close}, we obtain \eqref{rate_for_perp} with $C'=2C$. By choosing $\mathcal U$ sufficiently small and using~\cite[Proposition 2.6]{E73}, we obtain \eqref{linear_isom}.
Finally, \eqref{norm_bound_in_U} follows from Remark \ref{rem: U+-} and \eqref{cont_inv_bundle}.
\end{proof}

\subsection{Comparison lemmas of Jacobi fields}
Let $J$ be a nonzero Jacobi field along a unit speed geodesic $\gamma$. For any $t$ with $J(t)\neq 0$, define 
\begin{equation}\label{def: mu}
\mu_J(t):= \frac{1}{2}\frac{\left(\|J\|^2(t)\right)'}{\|J\|^2(t)}=\frac{g_{\gamma(t)}(J'(t), J(t))}{g_{\gamma(t)}(J(t), J(t))}.
\end{equation}
Notice that $\mu_J$ is invariant under scaling of the Jacobi field $J$. 

We will use the following comparison lemma from \cite{G75} many times in this paper.
\begin{lemma}[\hspace{1sp}\cite{G75}, Lemma 3]\label{lem: comp_thm}
Let $\gamma$ be a geodesic on a Riemannian manifold $M$ and let $J$ be a  perpendicular Jacobi field along $\gamma$. Assume that $f\colon\mathbb{R}\to\mathbb{R}$ is integrable on bounded sets and gives an upper bound on sectional curvature as follows
$$K(\textup{span}\{\gamma'(t), J(t)\})\leq f(t)$$ for all $t$. Let $s^*\in\R$ and let $u$ be a solution of $u''+fu=0$ with $u(s^*)=\|J\|(s^*), u'(s^*)\leq \|J\|'(s^*)$. Assume that $u(t)>0$ for $s^*<t\leq s^{**}$. Then for any $s^*<t\leq s^{**}$, $J(t)\neq 0$, and 
$$\mu_J(t)\geq u'(t)/u(t).$$
\end{lemma}

\begin{remark}\label{rem: behavior of Riccati solution}
    Let $u$ be a solution of $u''+fu=0$ where $f\colon \mathbb R\rightarrow\mathbb R$ is integrable on bounded sets. We define the logarithmic derivative of $u$ by $w=\frac{u'}{u}$. In particular, $u(t)=u(0)\exp{\int_0^t}w(s)ds$ and $w$ satisfies a first order non-linear equation 
    $$w'=-f-w^2.$$
    This equation shows that if $f\leq 0$ then the graph of $w$ crosses the graphs of $\sqrt{-f}$ and $-\sqrt{-f}$ horizontally, $w$ monotonically increases between them and decreases while above $\sqrt{-f}$ and below $-\sqrt{-f}$. 
    
    Thus, we get a good control on $w$ from below (know that it does not drop to $-\infty$ in the considered time) only if $w(0)$ is above $-\sqrt{-f}$. In particular, by Lemma~\ref{lem: comp_thm}, in that case we get a control on $\mu_J$.

    If $f>0$, then $w$ is monotonically decreasing so there is no good control from below.
\end{remark}

\begin{corollary}\label{cor: bdd_time_bdd_mu}
Let $(M, g)$ be a compact Riemannian manifold without conjugate points. For any $\tau>0$, there exists a constant $Q=Q(\tau, g)$ such that for any $v\in TM$ with $\gamma_v[0,\tau]\subseteq M$ the following holds. Let $J$ be a perpendicular Jacobi field along $\gamma_v$. If $\mu_J(0)>Q$ (we allow $\mu_J(0)=+\infty$), then  $\mu_J(t)>-Q$ for all $t\in[0,\tau]$. In particular, $J$ does not vanish on $(0,\tau]$.
\end{corollary}
\begin{proof}
Because $M$ is compact it admits an upper bound on sectional curvature $\tilde{K}^2$ and we can assume that $\tilde K>1$.

We argue by contradiction. Assume that there exists $\tau_0>0$ such that for any $n\in\mathbb N$, there exists $v_n\in TM$ with $\gamma_{v_n}[0,\tau_0]\subseteq M$ and perpendicular Jacobi field $J_n$ along $\gamma_{v_n}$ with $\mu_{J_n}(0)>n$ and $\mu_{J_n}(s_n)<-n$  for some $s_n\in[0,\tau_0]$. First, we prove that $s_n\geq \frac{2}{\tilde K}\tan^{-1}\frac{1}{\tilde K}$. If $J_n(0)\neq 0$, by applying Lemma \ref{lem: comp_thm} with $f\equiv \tilde{K}^2$, $s^*=0$, $u(0)=\|J_n\|(0)$, $u'(0)=n\|J_n\|(0)$, we have 
$$\mu_{J_n}(t)\geq \frac{u'(t)}{u(t)}=-\tilde{K}\tan\left(\tilde{K}t-\tan^{-1}\frac{n}{\tilde{K}}\right), \,\,\,\,\, t\in [0,\pi/2\tilde{K}].$$
Thus 
$$s_n\geq \frac{2}{\tilde K}\tan^{-1}\frac{n}{\tilde K}\geq \frac{2}{\tilde K}\tan^{-1}\frac{1}{\tilde K}$$
 If $J_n(0)= 0$, we may assume $\|J_n\|'(0)=1$, the solution to $u''+\tilde{K}^2u=0, u(0)=0, u'(0)=1$ is $u=\frac{1}{\tilde K}\sin(\tilde{K}t)$, thus 
$$\mu_{J_n}(t)\geq \frac{u'(t)}{u(t)}=\tilde K \cot(\tilde{K}t), \,\,\,\,\, t\in [0,\pi/\tilde{K}].$$
Hence $s_n\geq \frac{\pi}{2\tilde{K}}>\frac{2}{\tilde K}\tan^{-1}\frac{1}{\tilde K}$ since $\tilde{K}>1$. Thus, in either cases we have $s_n\geq \frac{2}{\tilde K}\tan^{-1}\frac{1}{\tilde K}$. In particular, 
$$\tau_0\geq \frac{2}{\tilde K}\tan^{-1}\frac{1}{\tilde K}$$

Without loss of generality, we assume that $\|J_n\|'(0)=1$ for all $n\in\mathbb N$. Thus,
\begin{equation*} 
\left(\|J_n\|^2\right)'(0)=2g_{\gamma_v(0)}(J'_n(0), J_n(0))\leq 2\| J'_n\|(0)\|J_n\|(0)= 2\|J_n\|(0).
\end{equation*}
Since $\mu_{J_n}(0)>n$, $\|J_n\|(0)< n^{-1}$. By taking a subsequence if necessary, we may assume that $v_n\to v, J'_n(0)\to w$ and  $s_n\to s\geq \frac{2}{\tilde K}\tan^{-1}\frac{1}{\tilde K}$ as $n\rightarrow\infty$ for some $v,w\in T^1M$ and $s\in[0,\tau]$. Let $J$ be a Jacobi field along $\gamma_v$ with $J(0)=0, J'(0)=w$. Then $J_n\to J$ as $n\to\infty$. On the other hand we have $\mu_J(0)=+\infty$ and $\mu_J(s)=-\infty$ thus $J(s)=0$. This contradicts to the fact that $M$ has no conjugate points.
\end{proof}

\subsection{The second fundamental form and the shape operator}\label{section: SFF}

In this section, we recall the definitions of the second fundamental form and the shape operator and their connection to sectional curvatures. See \cite{G94} for more details.

Let $S$ be an $(n-1)$-dimensional smooth manifold. Consider the product $(a,b)\times S$ with a Riemannian metric
\begin{equation*}
ds^2=dt^2+g_t,
\end{equation*}
where $t\in(a,b)$ and $g_t$ is a Riemannian metric on $S_t:=\{t\}\times S$. In particular, for any $\theta\in S$, we have that $\gamma(t)=(t,\theta)$, where $t\in(a,b)$, is a geodesic on $(a,b)\times S$. 

Define $\pi_s: \mathbb R\times S\to \mathbb R\times S$ by $\pi_s(t,\theta)=(t+s,\theta)$ for $\theta\in S$. The \textit{second fundamental form} on $S_t$ is a quadratic form given by:
\begin{equation}\label{def: second_fundamental_form}
\II_{S_t}(X, Y):=\frac{1}{2}\frac{d}{ds}\Big|_{s=0}g_{t+s}(d\pi_sX,d\pi_sY), \qquad\forall X, Y\in T_{(t,\theta)}S_t, \quad \forall t\in(a,b).
\end{equation}


The \textit{shape operator} $A(t, \theta): T_{(t,\theta)}S_t\to T_{(t,\theta)}S_t$ is the self-adjoint operator associated to $\II_{S_t}$ via
\begin{equation*}
\II_{S_t}(X, Y)=g_t(A(t, \theta)X, Y), \qquad\forall X, Y\in T_{(t,\theta)}S_t.
\end{equation*}


In particular, $A(t,\theta)$ is diagonalizable and its eigenvalues $\lambda_1(t,\theta)\leq \cdots\leq \lambda_{n-1}(t,\theta)$ are called the {\it principal curvatures} at $(t,\theta)$. The eigenvectors of $A(t,\theta)$ are called the  {\it principal directions} at $(t,\theta)$. Define $A(t)\colon S\rightarrow End(T_{(t,\cdot)}S_t)$ via $A(t)\theta:=A(t,\theta)$ and $\lambda_i(t)\colon S\to\mathbb{R}$ by $\lambda_i(t)\theta:=\lambda_i(t,\theta)$.  We say that $S_t$ is \textit{strictly convex} if $\lambda_1(t)>0$. Let $\lambda_{\max}(S_t) = \max\{\lambda_{n-1}(t,\theta)| \theta\in S_t\}$ and $\lambda_{\min}(S_t) = \min\{\lambda_1(t,\theta)|\theta\in S_t\}$.

For any vectors $X,Y\in T^1\left((a,b)\times S\right)$ such that $X$ and $Y$ are orthogonal, the sectional curvature of $\sigma_{X,Y}=\text{span}\{X, Y\}$ is defined by
\begin{equation}\label{sectional curvature}
K(\sigma_{X,Y})=\langle R(X,Y)Y,X\rangle,
\end{equation}
where $\langle\cdot,\cdot\rangle$ is the inner product corresponding to $ds^2$ and $R$ is the Riemann curvature tensor. In particular, 
\begin{equation*}
R(X,Y)Z = \nabla_X\nabla_Y Z-\nabla_Y\nabla_X Z-\nabla_{[X,Y]}Z, \quad\text{for any } X,Y, Z\in T\left((a,b)\times S\right),
\end{equation*}
where $[\cdot,\cdot]$ is the Lie bracket of vector fields.

Let $T=\partial/\partial t$. For any vector $X\in T^1S_t$, the sectional curvature of $\sigma_{X,T}=\text{span}\{X, T\}$ is given by
\begin{equation}\label{def: normal sectional}
K(\sigma_{X,T})=g_t(R(t)X, X),
\end{equation}
where $R(t):=-A(t)'-A(t)^2$ and $A(t)'(X):=\frac{d}{ds}\Big|_{s=0}\left(A(t+s)\theta\right)(d\pi_s X)$ for all $X\in T_{(t,\theta)}S_t$.

For any 2-plane $\sigma_{X,Y}=\text{span}\{X, Y\}\subseteq TS_t$ where $X, Y\in TS_t$, let $K^{\text{int}}(\sigma_{X,Y})$ be the intrinsic sectional curvature of $g_t$ at $\sigma_{X,Y}$. Then, the relation between $K^{\text{int}}(\sigma_{X, Y})$ and $K(\sigma_{X, Y})$ is given by Gauss' equation:
\begin{equation}\label{Gauss equation}
K(\sigma_{X, Y})=K^{\text{int}}(\sigma_{X, Y})-\frac{\II_{S_t}(X, X)\II_{S_t}(Y, Y)-\II_{S_t}(X, Y)^2}{|X\wedge Y|_t},
\end{equation}
where 
\begin{equation}\label{XwedgeY}
|X\wedge Y|_t=g_t(X, X)g_t(Y, Y)-g_t(X, Y)^2. 
\end{equation}
We have the following estimate on $K(\sigma)$ where 
$\sigma$ is a $2$-plane in $TS_t$.

\begin{lemma}\label{lem: bound_level_curv}
Assume $S_t$ is strictly convex. Then, for any 2-plane $\sigma\subseteq T_{(t,\theta)}S_t$,
$$K(\sigma)\leq K^{\text{int}}(\sigma)-\lambda_{\min}(S_t)^2.$$
\end{lemma}
\begin{proof}
Let $\{\tilde{e}_i\}_{i=1}^{n-1}$ be an orthonormal basis of $T_{(t,\theta)}S_t$ consisting of principal directions. Then, we have 
$$\II_{S_t}(\tilde{e}_i, \tilde{e}_j)=g_t(A(t,\theta)\tilde{e}_i, \tilde{e}_j)=\delta_{ij}\lambda_i(t,\theta),$$
where $\delta_{ij}$ is the Kronecker delta function. Let $X=\sum_{i=1}^{n-1} X_i\tilde{e}_i$ and $Y=\sum_{i=1}^{n-1} Y_i\tilde{e}_i$ be an orthonormal basis of $\sigma$. Then 
$$1=|X\wedge Y|_{t}=\sum_{i,j=1}^{n-1}X_i^2 Y_j^2-X_i Y_i X_j Y_j=\sum_{i<j}(X_i Y_j-X_jY_i)^2.$$
Thus, we have
\begin{eqnarray*}
& &\II_{S_t}(X, X)\II_{S_t}(Y, Y)-\II_{S_t}(X, Y)^2=\sum_{i,j=1}^{n-1}(X_i^2 Y_j^2-X_i Y_i X_j Y_j)\lambda_i(t,\theta)\lambda_j(t,\theta)\\
&=& \sum_{i<j}(X_i Y_j-X_jY_i)^2\lambda_i(t,\theta)\lambda_j(t,\theta) \geq  \lambda_{\min}(S_t)^2\sum_{i<j} (X_i Y_j-X_jY_i)^2=\lambda_{\min}(S_t)^2.
\end{eqnarray*}
By \eqref{Gauss equation}, we obtain $K(\sigma)\leq K^{\text{int}}(\sigma)-\lambda_{\min}(S_t)^2.$
\end{proof}

\section{Proofs of applications}\label{sec: proof_corollaries}

In this section we give proofs of Corollaries~\ref{cor_conf_class} and~\ref{thm_Livshits}.

\begin{proof}[Proof of Corollary~\ref{cor_conf_class}]
Denote by $\mu$ the normalized Riemannian volume on $\Sigma$ with respect to $g$. We can assume that $\int_\Sigma\rho^2d\mu\le 1$. (Otherwise we can exchange the roles of $g$ and $\rho^2g$ so that the conformal factor becomes $1/\rho^2$ and proceed in exactly same way.) 

We begin by applying Theorem \ref{thm_anosov_intro} and extend $(\Sigma, g)$ to a closed Anosov manifold $(\Sigma^{ext}, g^{ext})$. We also extend $\rho$ to $\rho^{ext}\colon \Sigma^{ext}\to\R$ by 1. Denote by $\mu^{ext}$ the normalized Riemannian volume on $(\Sigma^{ext}, g^{ext})$.

Assume $\rho^{ext}$ is not 1 everywhere on $\Sigma^{ext}$. Then by Cauchy-Schwartz inequality we have
$$
\int_{\Sigma^{ext}}\rho^{ext}d\mu^{ext}<1.
$$
Now following~\cite[Theorem 2]{K88}, we apply Birkhoff ergodic theorem and Anosov closing lemma to produce a unit speed geodesic $\gamma$ which approximates volume measure sufficiently well so that
$$
\int_{\gamma}\rho^{ext}(\gamma(t))dt<length(\gamma, g).
$$
Let $c$ be a connected component of $\gamma\cap\Sigma$. Denote by $c'$ the geodesic segment for $\rho^2g$ with the same entry and exit point as $c$. The universal cover $\tilde\Sigma$ equipped with the lift of $\rho^2g$ does not have conjugate points. Hence the segment $c'$ is the global minimizer. Thus
$$
\int_{c}\rho^{ext}(c(t))dt=length(c, \rho^2g)\ge length(c', \rho^2g)=length(c, g)
$$
where the last equality is due to the lens data assumption. By applying this inequality to each connected component of $\gamma\cap\Sigma$ and noting
 that $\rho^{ext}=1$ outside $\Sigma$ we obtain
$$
\int_{\gamma}\rho^{ext}(\gamma(t))dt\ge length(\gamma, g),
$$
 which gives a contradiction. Hence $\rho^{ext}=1$.
\end{proof}

\begin{remark} Using a local argument it is not hard to show that $\rho|_{\partial\Sigma}=1$. However, note that in the above proof we do not need to consider an extension of $g'$ and, in principle, $\rho$ is allowed to be discontinuous on the boundary of $\Sigma$.
\end{remark} 

\begin{proof}[Proof of Corollary~\ref{thm_Livshits}]
We begin by applying our main result to extend $X$ to an Anosov vector field, which we continue denote by $X$ on $\Sigma^{ext}\supset\Sigma$. Then we extend $\beta$ by the zero function. Because $C^r$-jet of $\beta$ vanishes on the boundary this extension remains $C^r$. 

For any periodic geodesic $\gamma$ which intersects boundary of $\Sigma$, we have 
$$
\int_\gamma \beta dt=0
$$
from the assumption of the corollary. Further we also have the following
\begin{lemma}
\label{lemma_aprx}
If $\gamma$ is a periodic geodesic in the interior of $\Sigma$ then
$$
\int_\gamma \beta dt=0.
$$
\end{lemma}

Assuming the lemma we can easily finish the proof by applying the Livshits Theorem~\ref{Liv} to $\beta$ and $X$ to obtain a $C^{r_-}$ solution $u\colon\Sigma^{ext}\to\mathbb R$ to the cohomological equation $Xu=\beta$.
To see that $u|_{\partial(T^1\Sigma)}=0$, pick a dense geodesic which intersect $\partial(T^1\Sigma)$ in a dense sequence of points $\{v_n\}_{n\in\mathbb Z}$. Because the integral of $\beta$ from $v_n$ to $v_{n+1}$ vanishes, by Newton's formula we have that $u(v_n)=u(v_{n+1})=const$ for all $n$. Hence, after subtracting the constant we indeed have $u|_{\partial(T^1\Sigma)}=0$.
\end{proof}
To finish the proof of Corollary~\ref{thm_Livshits}, we need to establish the lemma. This lemma is established using a standard shadowing argument.
\begin{proof}[Proof of Lemma~\ref{lemma_aprx}]
Recall that the trapped set $\Lambda\subset\Sigma$ consists of all geodesics which are entirely contained in the interior of $\Sigma$. In particular, $\gamma\subset\Lambda$. Without loss of generality, we may assume that $\Sigma$ is connected since $\gamma$ lies in one of the connected components of $\Sigma$.

We begin by observing that $\Lambda$ has a local product structure. Indeed, given a pair of sufficiently close points $x,y\in\Lambda$ the ``heteroclinic point" $[x,y]= W^s(x,\eps)\cap W^{0u}(y,\eps)$ stays close to the orbit of $x$ in the future and close to the orbit of $y$ in the past and, hence, remains in the interior of $\Sigma$ as well.

The first step of the proof is show that $\Lambda$ is nowhere dense. Assume that $\Lambda$ has non-empty interior $int(\Lambda)$. Let $\bar\Lambda$ be the closure of $int(\Lambda)$. It is easy to see that $int(\Lambda)$ and $\bar\Lambda$ still have a local product structure. (Hyperbolic set $\bar\Lambda$ could be a proper subset of $\Lambda$, for example, when $\Lambda$ has an isolated periodic orbit.) Note that $\bar\Lambda$ has positive volume. The restriction of the Sasaki volume to $\bar \Lambda$ is an ergodic measure. Therefore, by ergodicity, there exists a point $p\in int(\Lambda)$ whose forward orbit and backward orbits are both dense in $int(\Lambda)$ and, hence, are also dense in $\bar\Lambda$. Because $p$ is in the interior we have $W^s(p,\eps)\cup W^u(p,\eps)\subset\bar\Lambda$ for a sufficiently small $\eps>0$. Then for any $x\in\bar\Lambda$ we can pick forward iterates of $p$ which converge to $x$ and, hence, because $\bar\Lambda$ is closed and $W^u(p,\eps)$ expands, we have $W^u(x)\subset \bar\Lambda$. In the same way, by considering backwards orbit of $p$ we also have $W^s(x)\subset \bar\Lambda$. Finally, from the local product structure, for sufficiently small $\varepsilon>0$ we have $$\{x\}\neq\mathcal P_{\varepsilon}(x) = \{[y,z]=W^s(y,\varepsilon)\cap W^{0u}(z,\varepsilon)\in\bar\Lambda \,|\, y\in W^{0u}(x,\varepsilon), z\in W^{s}(x,\varepsilon)\}\subset\bar\Lambda.$$ In particular, $\mathcal P_{\varepsilon}(x)$ contains a neighborhood of $x$. Thus, we conclude that $x$ in an interior point of $\bar\Lambda$. This gives that the closed set $\bar\Lambda$ is also open which gives a contradiction because $\bar\Lambda$ is a proper subset of $\Sigma$.

Now we can use an approximation argument to show that $\int_\gamma\beta dt=0$. Let $p\in\gamma$ and let $q\in\gamma'$ be a point which is $\delta$-close to $p$ on a periodic geodesic $\gamma'$ which intersects the boundary of $\Sigma$. Existence of such a point $q$ follows from density of periodic orbits and the fact that $\Lambda$ is a closed nowhere dense set.

We now form a pseudo-orbit by pasting $\gamma$ and $\gamma'$ together and using Anosov closing lemma to produce a periodic orbit $\alpha$ which passes close to $[p,q]$ and first shadows $\gamma$ and then $\gamma'$; see Figure~\ref{fig: shadowing}. Clearly, such $\alpha$ intersects the boundary of $\Sigma$ as well and, hence, $\int_\alpha\beta dt=0$. Orbit $\alpha$ can be partitioned into 3 segments: one which shadows $\gamma$, one which shadows $\gamma'$  and a short remainder segment which appears due to joint non-integrability of strong foliations. More precisely, we let $\alpha=\alpha_1\cup\alpha_2\cup\alpha_3$, where $\alpha_1$ has the same length as $\gamma$ and relates to $\gamma$ via unstable-stable holonomy. Segment $\alpha_1$ is followed by $\alpha_2$ has the same length as $\gamma'$ and relates to $\gamma'$ via unstable-stable holonomy. Note that if we want the starting point of $\alpha_2$ to be related to $q$ via unstable-stable holonomy (as indicated on the figure) then we might need to reposition $q$ along $\gamma'$ to achieve that. Finally, the remaining segment $\alpha_3$ has length $<3\delta$ by application of triangle inequality. (For simplicity, we assume that $|\alpha|>|\gamma|+|\gamma'|$, if that's not the case then $\alpha_1$ and $\alpha_2$ would overlap and $\alpha_3$ would the the overlap; the same proof works in this case.)

\begin{figure}[h]
\centering
\includegraphics[scale=0.45]{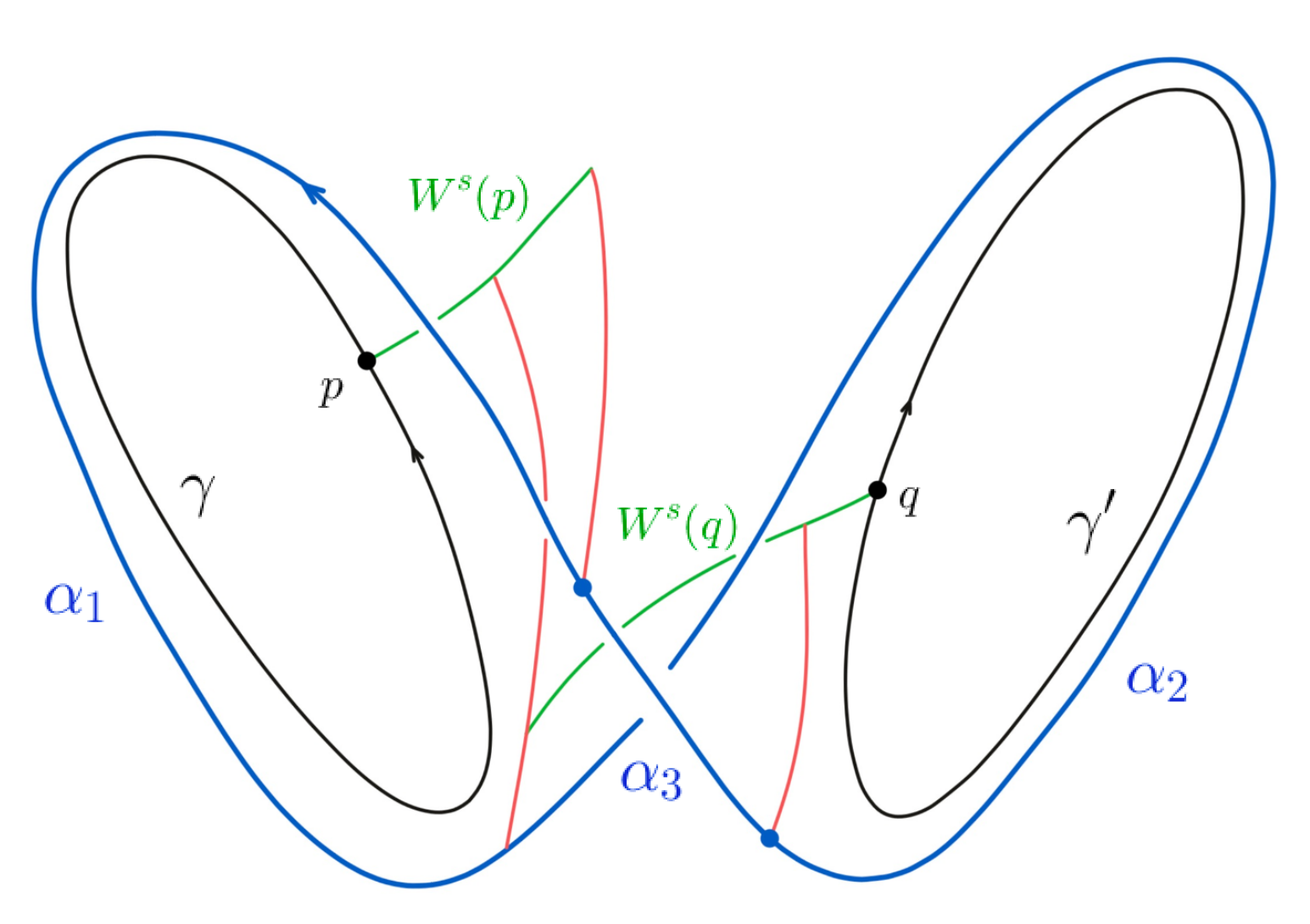}
\caption{Shadowing of $\gamma$ and $\gamma'$. Here we use green (resp. red) curves to denote stable (resp. unstable) manifolds.}
\label{fig: shadowing}
\end{figure}

By the standard ``exponential slacking" argument which is used in the proof of the Livshits Theorem~\cite{L71} we have
$$
\left|\int_\gamma\beta dt-\int_{\alpha_1}\beta dt   \right|\le\delta Lip(\beta)
$$
and
$$
\left|\int_{\gamma'}\beta dt-\int_{\alpha_2}\beta dt   \right|\le\delta Lip(\beta)
$$
where $Lip(\beta)$ is the Lipschitz constant of $\beta$.
\begin{remark} For the first difference an obvious crude upper bound $\delta Lip(\beta)|\gamma|$ would suffice. However for $\gamma'$ the above better bound is needed because the length of $\gamma'$ goes to $+\infty$ as $\delta\to 0$. 
\end{remark}
Because the end-points of $\alpha_3$ are $\delta$-close to $p$ and $q$ we also have
$$
\left|\int_{\alpha_3}\beta dt  \right|\le 3\delta\|\beta\|_{C^0}
$$
Also recall that $\int_{\gamma'}\beta dt=\int_{\alpha}\beta dt=0$. Putting these together we have
\begin{multline*}
\left|\int_{\gamma}\beta dt  \right|\le\delta Lip(\beta)+\left|\int_{\alpha_1}\beta dt  \right| \le \delta Lip(\beta)+\left|\int_{\alpha}\beta dt  \right| +\left|\int_{\alpha_2}\beta dt  \right| +\left|\int_{\alpha_3}\beta dt  \right|  \\
\le \delta Lip(\beta)+\left|\int_{\gamma'}\beta dt  \right|+\left|\int_{\gamma'}\beta dt-\int_{\alpha_2}\beta dt   \right|+3\delta\|\beta\|_{C^0}\le 2\delta Lip(\beta)+3\delta\|\beta\|_{C^0}
\end{multline*}
Taking $\delta\to 0$ we obtain $\int_\gamma\beta dt=0$.
\end{proof}

\section{A Jacobi estimate for geodesics which enter a domain with hyperbolic trapped set}\label{sec: mu_estimate_domain}

Following the outline of the proof (Section~\ref{outline of the proof}), we want to control the growth rates of the logarithm of the square of the norm of nonzero perpendicular Jacobi fields for the constructed compact Riemannian manifold. Consider a geodesic $\gamma$ and let $\tau_{M,\max}(\gamma)$ be the length of a maximal time interval so that the geodesic is in the given Riemannian manifold $M$ with boundary. In the presence of a trapped set, $\tau_{M,\max}(\gamma)$ can be arbitrarily large as a geodesic can be in the trapped set or accumulate for arbitrarily long time on it. Since the trapped set is hyperbolic, we can show that we have a ``good" control on the growth rates of the logarithm of the square of the norm of nonzero perpendicular Jacobi fields in a neighborhood of the trapped set. The precise result is the following proposition.

\begin{proposition}\label{prop: highent_highexit}
Let $(M, g)$ be a manifold with boundary. Assume that $(M, g)$ has no conjugate points and a (possibly empty) hyperbolic trapped set $\Lambda$. Then, there exists constants $Q_M>0$ and $ C_M>0$, which depend only on $M$, such that  for any $v\in\partial_-$ and a perpendicular Jacobi field $J$ along $\gamma_v$ with $\mu_J(0)>Q_M$,  $J$ does not vanish as long as $\gamma_v$ lies in $M$. Moreover, the following properties hold

(1) If $v\in \Gamma_-$, then $\|J\|(t)\to\infty$ as $t\to\infty$.

(2) If $v\notin \Gamma_-$, then $\mu_J(l_g(v))>-Q_M$ and $\int_0^{l_g(v)}\mu_J(\tau)d\tau\geq -C_M$. 

(3) For any sufficiently small $\delta>0$, let $M_{-\delta}:=\{x\in M| dist_g(x,\partial M)\geq \delta\}$. Then, (1) and (2) remain valid with the same $Q_M$ and $ C_M$ if we replace $M$ with $M_{-\delta}$.

\end{proposition}

In order to prove Proposition \ref{prop: highent_highexit} we need to analyze the behavior of Jacobi fields $J$ near the hyperbolic trapped set $\Lambda$.

\subsection{Neighborhood of hyperbolic trapped set}
For any $T\geq 0$, let 
$$\mathcal U_T(M) := T^1M-\bigcup\limits_{-T\leq t\leq T}\phi^t(\partial T^1M).$$ 
It is clear that $\mathcal U_T(M_{-\delta})\subseteq \mathcal U_{T+\delta}(M)\subseteq \mathcal U_{T}(M)$. Moreover the following simple lemma shows that $\mathcal U_T(M)\to\Lambda$ as $T\to \infty$.

\begin{lemma}\label{lemma_nbh_of_trapping_set}
For any $\eta>0$ there exists $T_0=T_0(\eta)$ such that $\mathcal O_{\eta}(\Lambda)\supset\mathcal U_{T_0}$, where $\mathcal O_{\eta}(\Lambda)$ is the open $\eta$-neighborhood of $\Lambda$ in the Sasaki metric.
\end{lemma}

\begin{proof}
Notice that for any $T\geq 0$ we have $\mathcal U_T$ is an open set and $\Lambda\subset \mathcal U_T$. Assume that the conclusion of the lemma does not hold. Then, there exists $\eta_0>0$ such that for any $n\in\mathbb N$ we have $\mathcal O_{\eta_0}(\Lambda)\not\supset \mathcal U_n$. In particular, for any $n\in\mathbb N$ there exists $x_n\in T^1M$ such that $x_n\in \mathcal U_n - \mathcal O_{\eta_0}(\Lambda)$. Moreover, $\mathcal U_{n+1}\subset\mathcal U_{n}$ for any $n\in\mathbb N$, and, by the definition of the trapped set, we have $\Lambda=\bigcap\limits_{n\in\mathbb N}\mathcal U_{n}$. 

By the compactness of $T^1M$, we obtain that there exists $x\in T^1M$ such that $x_n\rightarrow x$  in the Sasaki metric as $n\rightarrow +\infty$. Moreover, since $x_n\notin \mathcal O_{\eta_0}(\Lambda)$, we have that $x\notin \mathcal O_{\eta_0}(\Lambda)$. Also, $\overline{\mathcal O_{\frac{\eta_0}{2}}(\Lambda)}=\bigcap \limits_{n\in\mathbb N}\overline{\mathcal O_{\frac{\eta_0}{2}}(\mathcal U_{n})}$. In particular, there exists $j\in\mathbb N$ such that $x\in T^1M - \overline{\mathcal O_{\frac{\eta_0}{2}}(\mathcal U_i)}$ for any $i\geq j$. Thus, we obtain the contradiction to the fact that $x_n\rightarrow x$ as $n\rightarrow +\infty$ because for any $i\geq j$ we have $x_i\in \mathcal U_i$, so the distance between $x$ and $x_i$ is at least $\frac{\eta_0}{2}$.
\end{proof}

\subsection{Invariant Jacobi fields near $\Lambda$}

Let $\mathcal{U}$ be the open neighborhood as in Lemma \ref{lem: perp_component} with constant $C'$. We pick $T_0$ satisfying $\mathcal{U}_{T_0}\subseteq \mathcal{U}$ using Lemma  \ref{lemma_nbh_of_trapping_set}. For each $v\in \mathcal{U}$ and $w\in v^\perp$, let $\xi^{\sigma}_w$ be the vectors defined in Lemma \ref{lem: perp_component}(\ref{linear_isom}) and denote by $J^{\sigma}_{w}:=J_{\xi^{\sigma}_w}$. We have 
\begin{equation}\label{eq: 2nd_fund_form}
(J^{\sigma}_{w})'(t)=U_{\phi_t v}^{\sigma}J^{\sigma}_{w}(t).
\end{equation}
By Lemma \ref{lem: perp_component}\eqref{norm_bound_in_U}, there exists $L>0$, which is independent of $v$, such that 
$|U^\sigma_{v}|\leq L \text{ for all } v\in \mathcal{U}.$ Together with \eqref{eq: sasaki_jacobi} and \eqref{eq: 2nd_fund_form} we know that whenever $\phi^tv\in \mathcal{U}$ we have
\begin{equation}\label{eq: U_bound}
|d\phi_t(\xi^{\sigma}_w)|^2=\|J^{\sigma}_w\|^2(t)+\left(\|(J^{\sigma}_w)\|'(t)\right)^2\leq (1+L^2)\|J_w^{\sigma}\|^2(t),
\end{equation}
for all $w\in v^{\perp}$. Here $|\cdot|$ is the Sasaki norm defined in Section \ref{section: geometry tangent bundle}. Notice that the constants $C', L, \eta$ depend only on $\mathcal{U}$. 


\subsection{Decompostition of Jacobi fields near $\Lambda$}

Let $J$ be a perpendicular Jacobi field along $\gamma_v$ for some $v\in\partial^-$.  Let $T_0$ be the constant in Lemma \ref{lemma_nbh_of_trapping_set}.  When $l_g(v)>2T_0$,  let $\xi\in T_{\phi_{T_0}(v)}T^1M$ be the tangent vector at $\phi_{T_0}(v)$ with $J(t+T_0)=J_{\xi}(t)$. Since $\phi_{t}(v)\in \mathcal U$ for $t\in [T_0, l_g(x,v)-T_0]$, by Lemma \ref{lem: perp_component}, we can decompose $\xi$ as
$$\xi=\xi^s+\xi^u,$$
where $\xi^{\sigma}\in E^{\sigma}_{\perp}(\phi_{T_0}(v))$ for $\sigma= s,u$.
This decomposition can be represented in terms of Jacobi fields as follows:
$$J(t)=J^s(t-T_0)+J^u(t-T_0), \forall t\in [T_0, l_g(v)-T_0],$$
where $J^{\sigma}=J_{\xi_{\sigma}}\in \mathcal{J}^{\sigma}_{\perp}(\phi_{T_0}(v))$. The following proposition shows that the unstable component of $\xi$ cannot be too small when $\mu_J(0)$ and $l_g(v)$ are sufficiently large.
 
\begin{proposition}\label{prop: big_unstable_comp}
Assume the sectional curvature of $M$ is bounded from below by $-k^2$. Let $Q(T_0,g)$ be the constant defined in Corollary \ref{cor: bdd_time_bdd_mu}. Then there exists $D, \zeta>0$ depending on $\Lambda$ and $\mathcal{U}$ such that for any $v\in\partial_-$ with $l_g(v)>2T_0+D$, and any perpendicular Jacobi field $J$ along $\gamma_v$ with $\mu_J(t)>\max\{k+1, Q(T_0,g)\}$  for some $t\in[0,T_0]$, we have $|\xi^u|\geq \zeta |\xi|$.
\end{proposition} 

\begin{proof}
We argue by contradiction. Assume that we can find $t_n\in[0,T_0]$, $v_n\in\partial_-$ with $l_g(v_n)\to\infty$ and $J_n$ perpendicular Jacobi fields along $\gamma_{v_n}$ with $\mu_{J_n}(t_n)>k+1$ but at the same time $|\xi^u_n|<\frac{1}{n} |\xi_n|$. We may assume $t_n\to t$, $v_n\to v$ by passing to a subsequence and it is clear that $\gamma_v$ stays in $\mathcal{U}_{T_0}$ for $t\geq T_0$. In particular, $v\in \Gamma_-$. 

By definition of $Q(T_0,g)$, $J_n(t)\neq 0$ for all $n$ and $t\in[0,T_0]$. Without loss of generality we assume that $|\xi_n|=1$ for all $n$ thus $J_n\to J$ for some Jacobi field $J$ along $\gamma_v$. By Lemma \ref{lem: perp_component} the invariant bundles depend continuously on the base vectors, thus the projection to invariant components of Jacobi fields through $\mathcal U$ is continuous. Hence we have $J(t)=J^s(t-T_0)$ for $t\geq T_0$. Since $\gamma_v$ stays in $M$ for $t\geq T_0$, we also have $|\mu_J|\leq k$ by \cite[Proposition 2.11]{E73}.  On the other hand, since $\mu_{J_n}(t_n)>k+1$ for all $n$ and $J_n\to J, t_n\to t$, we have $\mu_{J}(t)\geq k+1>k$ which provides a contradiction.
\end{proof}

\begin{proof}[Proof of Proposition \ref{prop: highent_highexit}]
Take $\widetilde{T}>D$ so that 
$$ \zeta^2\left( \frac{C'^2}{2(1+L^2)}e^{\lambda \widetilde{T}}-C'^2e^{-\lambda \widetilde{T}}\left(2+\frac{2}{\zeta^2}\right) \right)>1.$$
It is clear that $\widetilde{T}$ also depends only on $\Lambda$ and $\mathcal{U}$. We take 
$$Q_M:=\max\{k+1, Q(2T_0+\widetilde{T},g)\}$$
 with $k$ as in Proposition \ref{prop: big_unstable_comp} and $Q$ given by Corollary \ref{cor: bdd_time_bdd_mu}.

First assume that $l_g(v)\geq 2T_0+\widetilde{T}$. If $\mu_J(0)\geq Q_M$, by Proposition \ref{prop: big_unstable_comp} and the parallelogram law, 
$$|\xi^s|^2=|\xi-\xi^u|^2\leq 2|\xi|^2+2|\xi^u|^2\leq \left(2+\frac{2}{\zeta^2}\right)|\xi^u|^2.$$
For all $t\in [\widetilde{T},l_g(v)-2T_0]$, by Proposition \ref{prop: big_unstable_comp},  (\ref{eq: U_bound}) and definition of hyperbolicity we have 
\begin{eqnarray}\label{eq: J_growth_hyp_nbd}
& &\|J\|^2(t+T_0)\geq  \frac{1}{2}\|J^u\|^2(t)-\|J^s\|^2(t) \geq  \frac{1}{2(1+L^2)}|d\phi_t(\xi^u)|^2-|d\phi_t(\xi^s)|^2 \nonumber\\
&\geq &\frac{C'^2e^{\lambda t}}{2(1+L^2)}|\xi^u|^2-C'^2e^{-\lambda t}|\xi^s|^2\geq \left( \frac{C'^2e^{\lambda t}}{2(1+L^2)}-C'^2e^{-\lambda t}\left(2+\frac{2}{\zeta^2}\right) \right) |\xi^u|^2 \nonumber\\
&\geq & \zeta^2\left( \frac{C'^2}{2(1+L^2)}e^{\lambda t}-C'^2e^{-\lambda t}\left(2+\frac{2}{\zeta^2}\right) \right) \|J\|^2(T_0).
\end{eqnarray}
Hence we finishes the proof of item (1). 

When $v\notin\Gamma_-$ estimate (\ref{eq: J_growth_hyp_nbd}) and our choice of $\widetilde{T}$ imply that $\|J\|(l_g(v)-T_0)>\|J\|(T_0)$, which can be written as
$$\int_{T_0}^{l_g(v)-T_0}\mu_J(\tau)d\tau\geq 0.$$
Moreover, we have $\mu_{J}(t)>-Q_M$ for all $t\in[l_g(v)-T_0, l_g(v)]$. Otherwise by reversing time, applying Proposition \ref{prop: big_unstable_comp} and repeating an argument similar to the above argument, we have $\|J\|(l_g(v)-T_0)<\|J\|(T_0)$, contradiction.

Hence when $l_g(v)\geq 2T_0+\widetilde{T}$,  we have $\mu_J(l_g(v))>-Q_M$ and 
$$\int_0^{l_g(v)} \mu_J(\tau)d\tau \geq  \int_0^{T_0} \mu_J(\tau)d\tau+\int_{l_g(v)-T_0}^{l_g(v)} \mu_J(\tau)d\tau\geq -2T_0 Q_M, $$

If $l_g(v)\leq 2T_0+\widetilde{T},$ then by Corollary \ref{cor: bdd_time_bdd_mu} we have
$$\int_0^{l_g(v)} \mu_J(\tau)d\tau \geq  -(2T_0+\widetilde{T}) Q_M.$$

Thus by taking $C_M:=-(2T_0+\widetilde{T}) Q_M$ we finish the proof of (2). The only part left is (3). Recall that all the constant $C, L, \zeta, \widetilde{T}$ depend on $\Lambda$ and its neighborhood $\mathcal U$. By replacing $M$ with $M_{-\delta}$ we still can work on a smaller neighborhood of $\Lambda$ thus the same argument goes through without any change. Thus we have finished the proof of Proposition \ref{prop: highent_highexit}.
\end{proof}

\section{Deformation to negative sectional curvature}\label{construction 1}

In this section we consider a cylinder with a given metric on a neighborhood $\mathcal O$ of one of the boundaries, and extend it to a metric on the whole cylinder so that the sectional curvatures is arbitrarily negative outside a small neighborhood of $\mathcal O$. We provide bounds on both sectional curvatures (see Section~\ref{setup notation}, Proposition~\ref{prop: collar 1,2}, Lemma~\ref{lem: bound_level_curv_C12}) and the principle curvatures of the equidistant sets. In particular, all the equidistant sets are also strictly convex. See the precise formulation of the main result Proposition~\ref{collar 1,2} of this section which is proved using the mentioned curvature bounds. 

\subsection{The setup and notation}\label{setup notation}

We use notation from Section~\ref{section: SFF}.

Let $S$ be an $(n-1)$-dimensional smooth closed manifold. For $\eps>0$, consider the product $(-\eps,0]\times S$ with a Riemannian metric 
\begin{equation}\label{isometry_boundary}
g = dt^2+g_t,
\end{equation}
where $g_t$ is the Riemannian metric on the hypersurface $S_t=\{t\}\times S$. Assume $S_0$ is strictly convex and recall that $h=2\II_{S_0}$ is the positive definite second fundamental form at $t=0$. 
For any $\theta\in S$, since $h$ is symmetric, there exists an orthonormal basis $\{e_i\}_{i=1}^{n-1}$ of $g_0$ such that $h(e_i,e_i)=2\lambda_i(0,\theta)>0$ where $\lambda_i(0,\theta)$ is the $i$-th principal curvature at $(0,\theta)$.
Our goal now is to extend the metric in a controlled way for $t>0$.

More generally to setup terminology, we can consider a manifold of the form $[a,b]\times S$ with coordinates $(t,\theta)$ where $t\in[a,b]$ and $\theta\in S$.  We say that a tangent $2$-plane $\sigma$ at $(t,\theta)$ is \textit{orthogonal} to $S_t$ if $\sigma$ contains a normal vector to $S_t$. As a result, we define \textit{orthogonal sectional curvatures} of $[a,b]\times S$ as curvatures of tangent $2$-planes orthogonal to $S_t$ for some $t\in[a,b]$. 

Let $\rho\colon\mathbb R\rightarrow [0,1]$ be a non-increasing $C^\infty$ function such that $\rho\equiv 1$ on $(-\infty,0]$ and $\rho\equiv 0$ on $[1,\infty)$. For any $\ell>0$, a function $f_{\ell}\colon\mathbb R\rightarrow\mathbb R$ is given by 
$$
f_{\ell}(t)=\frac{e^{\ell t}-1}{\ell}
,\quad t\in\mathbb R.$$

\begin{remark}
For any metric $g'$ on $S_0$, we consider its push-forward to a metric $(\pi_t)_*g'$ on $S_t$ which we still denote by $g'$ using a slight abuse of notation.
\end{remark}

\subsection{Deformation of the metric}
We prove the following result assuming Proposition \ref{prop: collar 1,2} and Lemma ~\ref{lem: bound_level_curv_C12}.
\begin{proposition}\label{collar 1,2}
(Notation of Section~\ref{setup notation}). Let $h=2\II_{S_0}$. Consider the manifold $[0,1+\eps]\times S$ with Riemannian metric
$\tilde g_{\ell,\eps} = dt^2+\tilde g_t$, where 
\begin{equation*}
\tilde g_t=\rho(t-\eps)g_0+f_{\ell}(t)h\quad\text{ for all }\quad t\in[0,1+\eps]. 
\end{equation*}
Then, for any $M_0>0$ there exists $K_g=K_g(g)$ and $L_{neg}=L_{neg}(M_0, g, \eps, \rho)>0$ such that for any $\ell>L_{neg}$ the following holds:
\begin{enumerate}[label=(\alph*)]
\item\label{upper positive} All sectional curvatures of $\tilde g_{\ell,\eps}$ are bounded from above by $K_g$;
\item\label{upper negative from epsilon} All sectional curvature of $\tilde g_{\ell,\eps}$ on $[\eps,1+\eps]\times S$ are bounded from above by $-M_0$;
\item\label{convex in general} For all $t\in[0,1+\eps]$, $S_t$ is strictly convex. Moreover, the principal curvatures of $S_t$ for $t\in[0,\eps]$ are bounded below by $\lambda_{\min}(S_0)$.
\end{enumerate}
\end{proposition}

\begin{proof}
Recall that $h$ is positive definite.
Item \ref{convex in general} will be proved later in Proposition~\ref{prop: collar 1,2} \eqref{convexity}.

Let $\sigma$ be a tangent $2$-plane at $(t_0,\theta_0)\in[0,1+\eps]\times S$. If $\sigma$ is orthogonal   to $S_{t_0}$, then Proposition~\ref{prop: collar 1,2} \eqref{normal} implies that it satisfies \ref{upper positive} and \ref{upper negative from epsilon} for sufficiently large $L_{neg}$. Otherwise, $\sigma=\sigma_{X+aT,Y}$ with $a\geq 0$ and $\{X,Y\}$ being orthonormal in $T_{(t_0,\theta_0)}S_{t_0}$. 

Thus, by \eqref{sectional curvature}, the sectional curvature of $\sigma$ is given by
\begin{align}
K_{\ell,\eps}(\sigma_{X+aT,Y}) &= \frac{1}{1+a^2}\langle R_{\ell,\eps}(X+aT,Y)Y,X+aT\rangle\label{sectional mixed expression}\\
&= \frac{1}{1+a^2}\left(K_{\ell,\eps}(\sigma_{X,Y})+a^2K_{\ell,\eps}(\sigma_{Y,T})+2a\langle R_{\ell,\eps}(X,Y)Y,T\rangle\right),\nonumber
\end{align} 
where $R_{\ell,\eps}$ is the Riemann curvature tensor.


Assume that
$X=\sum_{i=1}^{n-1}X_i(\pi_{t_0})_*e_i$, $Y=\sum_{i=1}^{n-1}Y_i(\pi_{t_0})_*e_i$,  where $\{e_i\}_{i=1}^{n-1}$ is the orthonormal basis defined in Section \ref{setup notation}. Since $\{X,Y\}$ are orthonormal in $T_{(t_0,\theta_0)}S_{t_0}$,  we have 
\begin{equation}\label{eq: norm of X and Y}
\sum\limits_{i=1}^{n-1}\tilde{X}_i^2=\sum\limits_{i=1}^{n-1}\tilde{Y}_i^2=1, \,\,\,\sum\limits_{i=1}^{n-1}\tilde{X}_i\tilde{Y}_i=0,
\end{equation}
where 
\begin{equation}\label{eq: tilde_XY}
\tilde{X}_i:=X_i\sqrt{\rho(t_0-\eps)+2f_\ell(t)\lambda_i(0,\theta_0)},\,\,\, \tilde{Y}_i:=Y_i\sqrt{\rho(t_0-\eps)+2f_\ell(t_0)\lambda_i(0,\theta_0)}.
\end{equation}
In particular, by \eqref{XwedgeY},
\begin{equation}\label{eq: area_identity}
1=|X\wedge Y|_{t_0}=\sum_{i,j=1}^{n-1}\tilde{X}_i^2\tilde{Y}_j^2-\tilde{X}_i\tilde{Y}_i\tilde{X}_j\tilde{Y}_j=\sum_{i<j}(\tilde{X}_i \tilde{Y}_j-\tilde{X}_j\tilde{Y}_i)^2.
\end{equation}
By (\ref{eq: tilde_XY}), we know that
$$|X_i|=\frac{|\tilde{X}_i|}{\sqrt{\rho(t-\eps)+2f_\ell(t)\lambda_i(0,\theta_0)}}\leq \frac{|\tilde{X}_i|}{\sqrt{\rho(t-\eps)+2f_\ell(t)\lambda_{\min}(S_0)}}.$$
Moreover, we have
\begin{align}
|\langle &R_{\ell,\eps}(X,Y)Y,T\rangle|\label{eq: mixed term bound}
\\&=\left|\sum_{i,j,k=1}^{n-1}X_iY_jY_k R^0_{kij}\right|=\left|\sum_{k=1}^{n-1}Y_k\sum_{i<j}(X_iY_j-X_jY_i)R^0_{kij}\right|\nonumber
\\&\leq C_1f'_\ell(t_0)\sum_{k=1}^{n-1}|Y_k|\sum_{i<j}\left|X_iY_j-X_jY_i\right|\quad\text{by Lemma~\ref{lem: bdd_curv_tensor_C12}}\nonumber
\\&\leq  \frac{C_1f'_\ell(t_0)}{(\rho(t_0-\eps)+2f_\ell(t_0)\lambda_{\min}(S_0))^{3/2}}\sum_{k=1}^{n-1} |\tilde{Y}_k|\sum_{i<j}|\tilde{X}_i \tilde{Y}_j-\tilde{X}_j\tilde{Y}_i| \quad\text{by \eqref{eq: tilde_XY}},\nonumber
\\&\leq \frac{C_1f'_\ell(t_0)\sqrt{n}(n-1)}{\sqrt{2}(\rho(t_0-\eps)+2f_\ell(t_0)\lambda_{\min}(S_0))^{3/2}}\sqrt{\sum_{k=1}^{n-1} |\tilde{Y}_k|^2\sum_{i<j}|\tilde{X}_i \tilde{Y}_j-\tilde{X}_j\tilde{Y}_i|^2}\nonumber
\\&= \frac{C_1f'_\ell(t_0)\sqrt{n}(n-1)}{\sqrt{2}(\rho(t_0-\eps)+2f_\ell(t_0)\lambda_{\min}(S_0))^{3/2}} \quad \text{by \eqref{eq: norm of X and Y} and \eqref{eq: area_identity}}.\nonumber
\end{align}
where we have used the Cauchy-Schwartz inequality after using \eqref{eq: tilde_XY}.

Thus,
\begin{equation*}
|\langle R_{\ell,\eps}(X,Y)Y,T\rangle|\to 0 \text{ uniformly in } \sigma_{X,Y} \text{ and } t_0\in [\eps,1+\eps] \text{ as }\ell\to\infty.
\end{equation*}

Moreover, by Lemma~\ref{lem: bound_level_curv_C12}, we have that $K_{\ell,\eps}(\sigma_{X,Y})\rightarrow -\infty$ uniformly in $\sigma_{X,Y}$ and $t_0\in[\eps,1+\eps]$. By Proposition~\ref{prop: collar 1,2} \eqref{normal}, $K_{\ell,\eps}(\sigma_{Y,T})<0$ for large enough $L_{neg}$. Therefore, by \eqref{sectional mixed expression}, we obtain \ref{upper negative from epsilon} in the proposition for a sufficiently large  $L_{neg}$. 

Now we consider the case when $t_0\in[0,\eps]$.  By \eqref{eqn: lambdai_t_0_ep} we have 
$$\lambda_{\min}(S_{t_0})= \frac{f'_\ell(t_0)\lambda_{\min}(S_0)}{1+2f_\ell(t_0)\lambda_{\min}(S_0)}.$$
Thus
\begin{eqnarray}\label{eqn: bdd_cross_curv}
& &-\frac{1}{1+a^2}\lambda_{\min}(S_{t_0})^2+\frac{2a}{1+a^2}\frac{C_1\sqrt{n}(n-1)f'_\ell(t_0)}{\sqrt{2}(1+2f_\ell(t_0)\lambda_{\min}(S_0))^{3/2}}\nonumber\\
&\leq & -\frac{1}{(1+a^2)(1+2f_\ell(t_0)\lambda_{\min}(S_0))}\left(\lambda_{\min}(S_0)^2z^2-2a\frac{C_1\sqrt{n}(n-1)}{\sqrt{2}}z\right)\nonumber\\
&\leq& \frac{1}{(1+a^2)(1+2f_\ell(t_0)\lambda_{\min}(S_0))}\frac{a^2C_1^2n(n-1)^2}{\lambda_{\min}(S_0)^2}\leq \frac{C_1^2n(n-1)^2}{\lambda_{\min}(S_0)^2},
\end{eqnarray}
where $z={f'_\ell(t_0)}/{\sqrt{1+2f_\ell(t_0)\lambda_{\min}(S_0)}}$. Recall that by Proposition~\ref{prop: collar 1,2} \eqref{normal}, $K_{\ell,\eps}(\sigma_{Y,T})< 0$ for a sufficiently large  $L_{neg}$. Therefore,  by \eqref{sectional mixed expression}\eqref{eq: mixed term bound}\eqref{eqn: bdd_cross_curv} and Lemma \ref{lem: bound_level_curv_C12},  we have
\begin{eqnarray*}
 K_{\ell,\eps}(\sigma_{X+aT,Y})&=&\frac{1}{1+a^2}\left(K_{\ell,\eps}(\sigma_{X,Y})+a^2K_{\ell,\eps}(\sigma_{Y,T})+2a\langle R_{\ell,\eps}(X,Y)Y,T\rangle\right)\\
 &\leq & \frac{1}{1+a^2} K^{\text{int}}_{\max,[0,1]}-\frac{1}{1+a^2}\lambda_{\min}(S_{t_0})^2+\frac{2a}{1+a^2}\frac{C_1\sqrt{n}(n-1)f'_\ell(t_0)}{\sqrt{2}(1+2f_\ell(t_0)\lambda_{\min}(S))^{3/2}}\\
&\leq & \max\{0,K^{\text{int}}_{\max,[0,1]}\}+\frac{C_1^2n(n-1)^2}{\lambda_{\min}(S_0)^2}=:K_g,
\end{eqnarray*}
where $K^{\text{int}}_{\max,[0,1]}$ is defined in Lemma \ref{lem: bound_level_curv_C12}.
Hence, we obtain item \ref{upper positive} of the proposition.

\end{proof}

\subsection{Upper bound on orthogonal sectional curvatures}\label{sec: normcurvC1C2}

\begin{proposition}[setting of Proposition~\ref{collar 1,2}]\label{prop: collar 1,2}
For any $M_0>0$, there exists $L_1=L_1(M_0, \eps, g, \rho)>0$ such that the following holds:
\begin{enumerate}
\item\label{convexity} Hypersurfaces $S_t$ are strictly convex for all $\ell>L_1$ and all $t\in[0,1+\eps]$. Moreover, the principal curvatures of $S_t$ for $t\in[0,\eps]$ are bounded below by $\lambda_{\min}(S_0)$. Also, $\lambda_{\min}(S_t)\rightarrow\infty$ uniformly in $t\in[\eps,1+\eps]$ as $\ell\rightarrow\infty$.
\item\label{normal} Let $K^\perp_{\ell,\eps}(t)$ be the maximum sectional curvature among planes $\sigma_{X,T}$ on $([0,1+\eps]\times S, \tilde g_{\ell,\eps})$, where $X\in T^1S_t$. Then, for all $\ell>L_1$ and all $t\in[0,1+\eps]$, 
\begin{equation*}
K^\perp_{\ell,\eps}(t)\leq -M_0.
\end{equation*}
\end{enumerate}
\end{proposition}

\begin{proof}
For any $\theta\in S$, let $e_i^t\in T_{(t,\theta)}S_t$ be defined by $e_i^t = (\pi_t)_{*}e_i,$ where $\{e_i\}$ is the orthonormal basis in Section \ref{setup notation}.
By construction, $\{e_i^t\}_{i=1}^{n-1}$ is an orthogonal basis of $T_{(t,\theta)}S_t$ for $t\in[0,1+\eps]$. Thus, any $X\in T_{(t,\theta)}S_t$ can be written in the coordinates as $(X_1,\ldots,X_{n-1})^T$ with respect to $\{e_i^t\}_{i=1}^{n-1}$. In particular,
\begin{equation*}
\tilde g_t(X,Y) = X^TG(t,\theta)Y,
\end{equation*}
where
\begin{equation*}
G(t,\theta):=\text{diag}(\rho(t-\epsilon)+2f_{\ell}(t)\lambda_1(0,\theta),\cdots, \rho(t-\epsilon)+2f_{\ell}(t)\lambda_{n-1}(0,\theta)).
\end{equation*}

For any $t\in[0,1+\eps]$, the second fundamental form on $S_t$ is given by
\begin{equation*}
\II_{S_t}(X,Y)=\frac{1}{2}X^T\frac{\partial G}{\partial t} (t,\theta)Y,\quad \text{  } X,Y\in T_{(t,\theta)}S_t.
\end{equation*}

Recall that $A(t,\theta)$ is the matrix of the shape operator on $S_t$ with respect to the basis to $\{e_i^t\}_{i=1}^{n-1}$, i.e., the $i$-th column of $A(t,\theta)$ is the image of $e_i^t$ under the shape operator. Then, by the definition of the shape operator,
\begin{equation*}
A(t,\theta) = \frac{1}{2}\frac{\partial G}{\partial t}(t,\theta)G(t,\theta)^{-1}.
\end{equation*}
Therefore, the $i$-th eigenvalue of $A(t,\theta)$ is given by 
\begin{equation*}
\lambda_i(t,\theta) = \frac{1}{2}\frac{\eta_i'(t,\theta)+f'_{\ell}(t)}{\eta_i(t,\theta)+f_{\ell}(t)},  \quad \text{ where }\quad \eta_i(t,\theta):=\frac{\rho(t-\eps)}{2\lambda_i(0,\theta)}.
\end{equation*}
In particular, $\lambda_{\min}(S_t)\rightarrow\infty$ uniformly in $t\in[\eps,1+\eps]$ as $\ell\rightarrow\infty$.

Furthermore, for $t\in[0,\eps]$, we have 
\begin{equation}\label{eqn: lambdai_t_0_ep}
\lambda_i(t,\theta) = \frac{1}{2} \frac{f'_{\ell}(t)}{\eta_i(t,\theta)+f_\ell(t)}.
\end{equation}
 Therefore, if $\ell>2\lambda_{\max}(S_0)$ then 
 $$\lambda_i(t,\theta)= \frac{1}{2}\frac{\ell}{[\ell\eta_i(0,\theta)-1]e^{-t\ell}+1}\geq \lambda_i(0,\theta)\geq\lambda_{\min}(S_0).
$$

Thus, there exists $\tilde L_1=\tilde L_1(\eps, \lambda_{\min}(S_0), \lambda_{\max}(S_0), \rho)>2\lambda_{\max}(S_0)$ such that for all $\ell>\tilde L_1$ we have $\lambda_{\min}(S_t)>0$ for all $t\in[0,1+\eps]$ and $i=1,\ldots, n-1$. Thus, all hypersurfaces $S_t$ are strictly convex and the principal curvatures of $S_t$ for $t\in[0,\eps]$ are bounded below by $\lambda_{\min}(S_0)$ proving Proposition~\ref{prop: collar 1,2} \eqref{convexity} for $\ell> \tilde L_1$.

Moreover, 
$$
\frac{\partial}{\partial t}A(t,\theta) = \frac{1}{2}\frac{\partial^2 G}{\partial t^2}(t,\theta)G(t,\theta)^{-1}-2A(t,\theta)^2. \qquad \text{}\nonumber
$$
Hence
\begin{equation}
-\frac{\partial}{\partial t}A(t,\theta)-A(t,\theta)^2 = -\frac{1}{2}\frac{\partial^2 G}{\partial t^2}(t,\theta)G(t,\theta)^{-1}+A(t,\theta)^2.\label{R(t) in basis}
\end{equation}

Using \eqref{R(t) in basis}, we obtain that the eigenvalues of $R(t,\theta)$, which is given by the matrix of $R(t)\theta$ relative to $\{e_i^t\}_{i=1}^{n-1}$ (see \eqref{def: normal sectional} for definitions), for all $i\in\{1,\ldots, n-1\}$ are given by
\begin{equation}\label{eigenvalues of R(t)}
r_i(t,\theta)=-\frac{1}{2}\left[\frac{\eta_i''(t,\theta)+f''_\ell(t)}{\eta_i(t,\theta)+f_\ell(t)}\right]+\frac{1}{4}\left[\frac{\eta_i'(t,\theta)+f'_\ell(t)}{\eta_i(t,\theta)+f_\ell(t)}\right]^2.
\end{equation} 

By \eqref{def: normal sectional}, we obtain $K^\perp_{\ell,\eps}(t)=\max\limits_{\theta\in S} r_i(t,\theta)$.

For all $t\in[0,\eps]$, we have 
\begin{align*}
r_i(t,\theta) &= -\frac{1}{2}\left[\frac{f''_\ell(t)}{\eta_i(0,\theta)+f_\ell(t)}\right]+\frac{1}{4}\left[\frac{f'_\ell(t)}{\eta_i(0,\theta)+f_\ell(t)}\right]^2\\
&=-\frac{\ell^2}{2}\left[\frac{e^{\ell t}}{\ell\eta_i(0,\theta)-1+e^{\ell t}}\right]+\frac{\ell^2}{4}\left[\frac{e^{\ell t}}{\ell\eta_i(0,\theta)-1+e^{\ell t}}\right]^2\\
&=\frac{\ell^2}{4}\left[\left(\frac{e^{\ell t}}{\ell\eta_i(0,\theta)-1+e^{\ell t}}-1\right)^2-1\right]\\
&=\frac{\ell^2}{4}\left[\left(\frac{\ell\eta_i(0,\theta)-1}{\ell\eta_i(0,\theta)-1+e^{\ell t}}\right)^2-1\right]\leq \frac{\ell^2}{4}\left[\left(\frac{\ell\eta_i(0,\theta)-1}{\ell\eta_i(0,\theta)}\right)^2-1\right]\\
&= -\ell\lambda_i(0,\theta)+\lambda^2_i(0,\theta)\leq -\ell\lambda_{\min}(S)+\lambda^2_{\max}(S).
\end{align*}

We conclude that for all $\ell>({M_0+\lambda^2_{\max}(S_0)})/{\lambda_{\min}(S_0)}$ and for all $t\in[0,\eps]$ we have $K^\perp_{\ell,\eps}(t)\leq -M_0$.

Moreover, by \eqref{eigenvalues of R(t)}, since 
$$
\frac{f''_\ell(t)}{f_\ell(t)}=l\frac{f'_\ell(t)}{f_\ell(t)}=l^2\frac{e^{\ell t}}{e^{\ell t}-1}
$$
for $t\in[\eps,1+\eps]$,  as $\ell\to\infty$, we have 
$$r_i(t,\theta)=-\frac{1}{2}\left[\frac{\ell\eta_i''(t,\theta)+\ell^2e^{\ell t}}{\ell\eta_i(t,\theta)+e^{\ell t}-1}\right]+\frac{1}{4}\left[\frac{\eta_i'(t,\theta)+\ell e^{t\ell}}{\ell\eta_i(t,\theta)+e^{\ell t}-1}\right]^2\approx -\frac{\ell^2}{2}+\frac{\ell^2}{4}=-\frac{\ell^2}{4}.$$
Thus there exists $\tilde L_2=\tilde L_2(M_0,\eps, \lambda_{\max}(S_0), \lambda_{\min}(S_0), \rho)$ such that $K^\perp_{\ell,\eps}(t)\leq-M_0$ for all $\ell>\tilde L_2$ and $t\in[\eps,1+\eps]$.

Finally, taking $L_1 = \max\{\tilde L_1, \tilde L_2,2\lambda_{\max}(S_0),({M_0+\lambda^2_{\max}(S_0)}){\lambda_{\min}(S_0)}\}$ finishes the proof of Proposition~\ref{prop: collar 1,2} \eqref{normal}. 
\end{proof}

\subsection{Upper bound on level sectional curvatures}

\begin{lemma}[setting of Proposition~\ref{collar 1,2}]\label{lem: bound_level_curv_C12}
There exists a constant $L_2=L_2(\eps, \lambda_{\min}(S_0), \rho)$ such that for any $t\in [0,1+\eps]$, any tangent 2-plane $\sigma\subseteq T_{(t,\theta)}S_t$, and $\ell>L_2$, we have the following upper bound on the sectional curvature of $\tilde g_{\ell,\eps}$ at $\sigma$:
\begin{equation*}
K_{\ell,\eps}(\sigma)\leq K^{\text{int}}_{\max,[0,1]}-\lambda_{\min}(S_t)^2,
\end{equation*} where $K^{\text{int}}_{\ell,\eps}(\sigma)$ the intrinsic sectional curvature of $\tilde{g}_t$ at $\sigma$, $K^{\text{int}}_{\max}(\hat g)$ is the maximum sectional curvature on $(S,\hat g)$, and
\begin{equation*}
K^{\text{int}}_{\max,[0,1]} := \max\limits_{a\in[0,1]}\{K^{\text{int}}_{\max}(g_0+a h), K^{\text{int}}_{\max}(ag_0+h) \}.
\end{equation*}

Moreover, $K_{\ell,\eps}(\sigma)\rightarrow -\infty$ as $\ell\rightarrow\infty$ uniformly in $\sigma$ and $t\in[\eps, 1+\eps]$.
\end{lemma}

\begin{proof}
According to the proof of Proposition \ref{prop: collar 1,2}, there exists  a positive constant $\tilde L_1=\tilde L_1(\eps, \lambda_{min}(S_0), \lambda_{max}(S_0), \rho)$  such that for all $\ell>\tilde L_1$ we have $S_t$ is strictly convex for all $t\in[0,1+\eps]$. We take $L_2>\tilde L_1$ such that for any $l>L_2$, $f_{\ell}(\eps)>1$. By Lemma \ref{lem: bound_level_curv}, we have
$$K_{\ell,\eps}(\sigma)\leq K^{\text{int}}_{\ell,\eps}(\sigma)-\lambda_{\min}(S_t)^2.$$
Now we estimate $K^{\text{int}}_{\ell,\eps}(\sigma)$ from above. For any fixed $\ell>L_2$, if $f_\ell(t)\leq 1$, then $t<\eps$. Thus, $K^{\text{int}}_{\ell,\eps}(\sigma)\leq K^{\text{int}}_{\max,[0,1]}$. If $f_\ell(t)\geq 1$,  then
$$K^{\text{int}}_{\ell,\eps}(\sigma)=\frac{1}{f_\ell(t)} K^{\text{int}}\left(h+\frac{\rho(t-\eps)g(0)}{f_\ell(t)}, \sigma\right)\leq \frac{K^{\text{int}}_{\max,[0,1]}}{f_\ell(t)}\leq K^{\text{int}}_{\max,[0,1]}.$$
Thus, $K^{\text{int}}_{\ell,\eps}(\sigma)\leq K^{\text{int}}_{\max,[0,1]}$ for all $\sigma$ with $t\in[0,1+\eps]$. 

Furthermore, using Proposition~\ref{prop: collar 1,2} \eqref{convexity}, we obtain that $K_{\ell,\eps}(\sigma)\rightarrow -\infty$ uniformly in $\sigma$ and $t\in[\eps,1+\eps]$.
\end{proof}

\section{``Rounding" the metric}\label{construction 2}

In this section we consider a cylinder with given metrics on the boundaries. Then, we use a linear combination of those metrics on each equidistant set to define a metric of the form $\tilde g=dt^2+\tilde g_t$ on the whole cylinder so that it has the given metrics on the boundary. Then, by choosing an appropriate exponentially growing function $f_\ell$ of the distance to one of the boundaries, we can guarantee that a metric $\tilde g=dt^2+f_\ell(t)\tilde g_t$ has arbitrarily negative the sectional curvatures (see Propositions~\ref{prop: collar 3},~\ref{prop: level sectional 2}). We can guarantee that all the equidistant sets are also strictly convex. See the precise formulation in Proposition~\ref{prop: sectional 2}.

Our aim is to glue a given metric on the manifold with boundary with the standard hyperbolic metric. In regards of that, Proposition~\ref{prop: sectional 2} allows us to ``round up" the metric through the cylinder meaning have a non-conformal metric on one end of the cylinder and a conformal metric on the other end of it while having arbitrarily negative curvature and strict convexity of the equidistant sets.

\begin{proposition}\label{prop: sectional 2}
(Notation of Section~\ref{setup notation}).Let $h$ and $\hat h$ be Riemannian metrics on $S$. Consider the manifold $[0,1+\eps]\times S$ with Riemannian metric $\hat g_{\ell,\eps} = dt^2+\hat g_t$ where 
$$\hat g_t = f_\ell(t+1+\eps)\left(\rho(t)h+(1-\rho(t))\hat h\right),\,\,\,t\in[0,1+\eps].$$

Then, for any $M_0>0$ there exists $L_r = L_r(M_0,\eps,h,\hat h,\rho)>0$ such that for any $\ell>L_r$ the following holds:
\begin{enumerate}[label=(\alph*)]
\item\label{all negative} All sectional curvatures of $\hat g_{\ell,\eps}$ are bounded from above by $-M_0$;
\item\label{gen convexity 2} For all $t\in[0,1+\eps]$, $S_t$ is strictly convex.
\end{enumerate}
\end{proposition}

\begin{remark}
We will use Proposition~\ref{prop: sectional 2} in Proposition~\ref{prop: C11 extension} for $h=2\II_{S_0}$, where $S_0$ is from Section~\ref{setup notation}, and $\hat h$ being the standard round metric of curvature $1$ on a sphere. 
\end{remark}

\begin{proof}
The proof follows the same general approach as the proof of Proposition~\ref{collar 1,2} so we omit some of the details.

By Proposition~\ref{prop: collar 3} \eqref{convexity 2}, we have item \ref{gen convexity 2}.

Moreover, by Propositions ~\ref{prop: collar 3} \eqref{normal 2} and ~\ref{prop: level sectional 2}, we only need to prove \ref{all negative} for a tangent $2$-plane $\sigma$ at $(t_0,\theta_0)$ which is neither tangent nor orthogonal to $S_{t_0}$. Then, $\sigma=\sigma_{X+aT, Y}$ where $a>0$, $X,Y\in T^1_{(t_0,\theta_0)}S_{t_0}$ and $X,Y$ are orthogonal. 

For any $\theta\in S$, let $\{e_i\}_{i=1}^{n-1}$ be an orthonormal basis of $h$ which also diagonalizes $\hat h$ and let $\hat h(e_i,e_i)=\mu_i(\theta)$.
In particular, $X=\sum\limits_{i=1}^{n1}X_i(\pi_{t_0})_*e_i$ and $Y=\sum\limits_{i=1}^{n-1}Y_i(\pi_{t_0})_*e_i$ where for all $i\in\{1,\ldots,n-1\}$
\begin{equation}\label{bounds on coordinates 2}
|X_i|,|Y_i|\leq \frac{1}{f_\ell(t_0+1+\eps)^\frac{1}{2}\min\{1,\mu_{\min}(S)\}^\frac{1}{2}}.
\end{equation}

Using Lemma~\ref{lem: bdd_curv_tensor_C3} and \eqref{bounds on coordinates 2}, we obtain that for all $t_0\in[0,1+\eps]$
\begin{align*}
\left|\langle R(X,Y)Y,T\rangle\right| &\leq \frac{(n-1)^3D_{\hat h}}{\min\{1,\mu_{\min}(S)\}^\frac{3}{2}}\left(1+\frac{2(1+\mu_{\max}(S))}{2\mu_{\min}(S)}\right)\frac{f'_\ell(1+\eps)+f_\ell(1+\eps)M'_\rho}{f_\ell(1+\eps)^{\frac{3}{2}}}\\
&\rightarrow 0 \quad\text{as}\quad \ell\rightarrow\infty. 
\end{align*}

Thus, by \eqref{sectional curvature},\eqref{sectional mixed expression}, and applying Propositions~\ref{prop: collar 3} \eqref{normal 2} and ~\ref{prop: level sectional 2} for $M_0+1$ instead of $M_0$, we obtain that there exists $L_r=L_r(M_0, \eps, h, \hat h, \rho)>0$ such that for all $\ell>L_r$, we have $K_{\ell,\eps}(\sigma)\leq -M_0$ for all $t_0\in[0,1+\eps]$.
\end{proof}

\subsection{Upper bound on orthogonal sectional curvatures}

\begin{proposition}[setting of Proposition~\ref{prop: sectional 2}]\label{prop: collar 3}
For any $M_0>0$, there exists a constant $L_1=L_1(M_0, \eps, h, \hat h, \rho)>0$ such that the following holds:
\begin{enumerate}
\item\label{convexity 2} Hypersurfaces $S_t$ are strictly convex for all $t\in[0,1+\eps]$. Moreover, $\mu_{\min}(S_t)=\min\{\mu_{i}(t,\theta)|\theta\in S_t,\, i=1,\ldots,n-1\}\rightarrow\infty$ uniformly in $t\in[0,1+\eps]$ as $\ell\rightarrow\infty$ where $\{\mu_{i}(t,\theta)\}$ are principal curvatures of $S_t$.
\item\label{normal 2} Let $K^\perp_{\ell,\eps}(t)$ be the maximum sectional curvature among planes $\sigma_{X,T}$ on $([0,1+\eps]\times S, \hat g_{\ell,\eps})$, where $X\in T^1S_t$. Then, for all $\ell>L_1$ and all $t\in[0,1+\eps]$, 
\begin{equation*}
K^\perp_{\ell,\eps}(t)\leq -M_0.
\end{equation*}

\end{enumerate}
\end{proposition}

\begin{proof}
The proof follows the same approach as the proof of Proposition~\ref{prop: collar 1,2} so we omit some details.

For any $\theta\in S$, let $\{e_i\}_{i=1}^{n-1}$ be an orthonormal basis of $h$ such that $\hat h(e_i,e_j)=\mu_i(\theta)$. Let $\mu_{\max}(S) = \max\{\mu_i(\theta)| i\in\{1,\ldots, n-1\}, \theta\in S\}$ and similarly $\mu_{\min}(S) = \min\{\mu_i(\theta)| i\in\{1,\ldots, n-1\}, \theta\in S\}$.

For any $\theta\in S$, let $e_i^t\in T_{(t,\theta)}S_t$ be defined by $e_i^t = (\pi_t)_{*}e_i.$ By the construction, $\{e_i^t\}_{i=1}^{n-1}$ is an orthogonal basis of $T_{(t,\theta)}S_t$ for $t\in[0,1+\eps]$. Thus, any $X\in T_{(t,\theta)}S_t$ can be identified with the coordinate vector $(X_1,\ldots,X_{n-1})^T$ with respect to $\{e_i^t\}_{i=1}^{n-1}$. In particular,
\begin{equation*}
\tilde g_t(X,Y) = X^TG(t,\theta)Y,
\end{equation*}
where
\begin{equation*}
G(t,\theta):=f_\ell(t+1+\eps)\text{diag}(\rho(t)+(1-\rho(t))\mu_1(\theta),\cdots, \rho(t)+(1-\rho(t))\mu_{n-1}(\theta)).
\end{equation*}

For any $t\in[0,1+\eps]$, the $i$-th eigenvalue of $A(t,\theta)$ is given by 
\begin{equation*}
\mu_i(t,\theta) = \frac{1}{2}\left(\frac{f_\ell'(t+1+\eps)}{f_\ell(t+1+\eps)}+\frac{\rho'(t)(1-\mu_i(\theta))}{\rho(t)+(1-\rho(t))\mu_i(\theta)}\right).
\end{equation*}

Thus, there exists $\bar L_1=\bar L_1(\eps, \mu_{min}(S), \mu_{max}(S), \rho)>0$ such that for all $\ell>\bar L_1$ we have $\min\limits_{\theta\in S}\mu_i(t,\theta)>0$ for all $t\in[0,1+\eps]$ and $i=1,\ldots, n-1$. Thus, $S_t$ are strictly convex. Moreover, we have
\begin{equation}\label{unif_increase_mu_i}
\mu_{\min}(S_t) \to \infty \text{ uniformly on } t\in[0,1+\eps] \text{ as } l\to\infty.
\end{equation}

Using \eqref{R(t) in basis}, we obtain that the eigenvalues of $R(t,\theta)$ which is the matrix of $R(t)\theta$ in the basis $\{e_i^t\}_{i=1}^{n-1}$ (see \eqref{def: normal sectional} for definitions) for all $i\in\{1,\ldots, n-1\}$ are given by
\begin{multline}\label{eigenvalues of R(t) 2}
r_i(t,\theta)
=-\frac{1}{2}\Bigg[\frac{f''_\ell(t+1+\eps)}{f_\ell(t+1+\eps)}\\
+(1-\mu_i)\left(\frac{2f'_\ell(t+1+\eps)\rho'(t)}{f_\ell(t+1+\eps)(\rho(t)+(1-\rho(t))\mu_i)}+\frac{\rho''(t)}{\rho(t)+(1-\rho(t))\mu_i}\right)\Bigg]\nonumber\\
+\frac{1}{4}\left[\frac{f_\ell'(t+1+\eps)}{f_\ell(t+1+\eps)}+\frac{\rho'(t)(1-\mu_i(\theta))}{\rho(t)+(1-\rho(t))\mu_i(\theta)}\right]^2.
\end{multline}

Since $K^\perp_{\ell,\eps}(t)=\max\limits_{\theta\in S} r_i(t,\theta)$, there exists a constant $\bar L_2=\bar L_2(M_0,\eps, \mu_{\max}(S), \mu_{\min}(S), \rho)$ such that $K^\perp_{\ell,\eps}(t)\leq-M_0$ for all $\ell>\bar L_2$ and $t\in[0,1+\eps]$.

By taking $L_1 = \max\{\bar L_1, \bar L_2\}$, we prove Proposition~\ref{prop: collar 3}. 
\end{proof}

\subsection{Upper bound on level sectional curvatures}
\begin{proposition}\label{prop: level sectional 2}
Assume we are in the setting of Proposition~\ref{prop: sectional 2}. For any $M_0>0$, there exists a constant $L_2=L_2(M_0,\eps, h, \hat h, \rho)>0$ such that for any $\ell>L_2$, $t\in[0,1+\eps]$, and tangent $2$-plane $\sigma\subseteq T_{(t,\theta)}S_t$, we have $K_{\ell,\eps}(\sigma)\leq -M_0$.
\end{proposition}

\begin{proof}
For any 2-plane $\sigma\subseteq  T_{(t,\theta)}S_t$, we obtain
\begin{align*}
|K^{\text{int}}(\hat g_t,\sigma)| &= \left|\frac{1}{f_\ell(t+1+\eps)}K^{\text{int}}\left(\rho(t)h+(1-\rho(t))\hat h,\sigma\right)\right|\\
&\rightarrow 0 \quad {as }\quad \ell\rightarrow\infty\quad\text{uniformly in $\sigma$ and $t\in[0,1+\eps]$.}\nonumber
\end{align*}

By Proposition~\ref{prop: collar 3}\eqref{convexity 2}, for all $\ell>L_1$ and $t\in[0,\eps]$, we have $S_t$ is strictly convex. Moreover, $\mu_{\min}(S_t)\rightarrow\infty$ uniformly in $t\in[0,1+\eps]$ as $\ell\rightarrow\infty$. Thus, by Lemma \ref{lem: bound_level_curv},
$$K_{\ell,\eps}(\sigma)\leq K^{\text{int}}(\hat g_t, \sigma)-\mu_{\min}(S_t)^2\to -\infty \text{ as } \ell\to\infty \text{ uniformly in } \sigma \text{ and } t\in [0,1+\eps].$$

As a result, for any $M_0>0$  there exists a constant $L_2=L_2(M_0,\eps, h, \hat h, \rho)>0$ such that $K_{\ell,\eps}(\sigma)\leq -M_0$ for all $t\in[0,1+\eps]$, tangent $2$-plane $\sigma\subseteq T_{(t,\theta)}S_t$, and all $\ell>L_2$.
\end{proof}

\section{The $C^{1,1}$ and $C^{\infty}$ extensions}\label{sec: metric_ext}

The goal of this section is to construct a $C^{1,1}$-extension to the constant negative curvature of a given metric on the product of infinite ray and a sphere. In the second half of this section we will mollify the $C^{1,1}$ metric to obtain a $C^\infty$ metric while still controlling the curvature.

\subsection{$C^{1,1}$ extension to constant negative curvature}\label{C1 extension}
We use the notation introduced in Section~\ref{setup notation}. We also assume $\eps$ is small enough so that the principal curvatures of $S_t$ are at least $3\lambda_{\min}(S_0)/4$ for $t\in(-\eps, 0)$.

\begin{proposition}\label{prop: C11 extension}
(setting of Section~\ref{setup notation}) Assume $S$ is a sphere and $ds_{n-1}^2$ is the standard round metric of curvature 1 on $S$. Let $h=2\II_{S_0}$. For any $M_0>0$ and $M_1>0$ there exist $K_g=K_g(g)$ and $L=L(M_0, M_1,g,\eps,\rho)$ such that for any $\ell>L$ there exist $\kappa>M_1$ and $\tilde r>-2-2\eps$ with the following properties. Consider the manifold $(-\eps,\infty)\times S$ with the Riemannian metric $g^{\ell,\eps}=dt^2+g^{\ell,\eps}_t$ where
\begin{equation*}
g^{\ell,\eps}_t = \left\{\begin{aligned} &g_t, \qquad t\in(-\eps,0),\\
&\rho(t-\eps)g_0+f_{\ell}(t)h, \qquad t\in[0,1+\eps],\\
&f_{\ell}(t)\left(\rho(t-1-\eps)h+(1-\rho(t-1-\eps)ds^2_{n-1})\right), \qquad t\in[1+\eps,2+2\eps],\\
&\left(\frac{1}{\kappa}\sinh[\kappa(t+\tilde r)]\right)^2ds^2_{n-1}, \qquad t\in[2+2\eps,\infty).
\end{aligned}\right.
\end{equation*}
Then, the following holds:
\begin{enumerate}[label=(\alph*)]
\item $g^{\ell,\eps}$ is a $C^{1,1}$-metric which is $C^\infty$ if $t\neq 0, 2+2\eps$;
\item All hypersurfaces $S_t$ are strictly convex. Moreover, the principal curvatures of $S_t$ are at least $\lambda_{\min}(S_0)$ for $t\in(0,\eps)$;
\item All sectional curvatures of $g^{\ell,\eps}$ on $(-\eps,\eps)\times S$ are less than or equal to $K_g$;
\item All sectional curvatures of $g^{\ell,\eps}$ on $(\eps,2+2\eps)\times S$ are less than or equal to $-M_0$;
\item All sectional curvatures of $g^{\ell,\eps}$ on $(2+2\eps, \infty)\times S$ are $-\kappa^2$.
\end{enumerate}
\end{proposition}

\begin{proof}
Notice that $h$ is a Riemannian metric on $S$ as $S_0$ is strictly convex.

Because $f_l$ and $\rho$ are smooth the metric $g^{\ell,\eps}$ is smooth in each component. Via the choice of $f_l, \rho$ and $\kappa$ (in Lemma \ref{lemma: combine f with constant 2}), it is clear that $g^{\ell,\eps}$ is smooth at $t=1+\eps$ and $C^{1,1}$ at $t=0, 2+2\eps$. Thus we obtain (a). Moreover, Lemma \ref{lemma: combine f with constant 2} shows that there exists $\tilde{L}_1=\tilde{L}_1(M_1)$ such that for any $\ell>\tilde{L}_1$, the associated $\kappa$ is at least $M_1$. 
Item (c) follows from ~Proposition \ref{collar 1,2}\ref{upper positive}, while (e) follows from Lemma \ref{lemma: curvature for product}. 

Notice that the construction on $t\in[1+\eps,2+2\eps]$ is just a translation reparametrization of the metric in Proposition \ref{prop: sectional 2}. Thus (d) follows from ~Proposition \ref{collar 1,2}\ref{upper negative from epsilon} and Proposition \ref{prop: sectional 2}\ref{all negative}. Finally we get (b) via ~Proposition \ref{collar 1,2}\ref{convex in general}, Proposition \ref{prop: sectional 2}\ref{gen convexity 2} and the assumption of $\eps$ above this proposition.
\end{proof}

\subsection{Smoothing of the extension from Section~\ref{C1 extension}}\label{smooth extension}

We apply a technique developed in \cite{EK19} to smooth out the $C^{1,1}$ metric we obtained in Proposition~\ref{prop: C11 extension}. 

\begin{proposition}\label{prop: smoothify}
Consider $M_1>1$. Let $g^{\ell,\eps}$ be the Riemannian metric on $(-\eps,\infty)\times S$ from Proposition~\ref{prop: C11 extension} with $M_0=M^2_1$. Then, for any $\delta\in(0,\frac{\eps}{2})$ there exists $K_0=K_0(\ell, \eps)>0$ and a smooth Riemannian metric $\tilde g^{\ell,\eps}$ on $(-\eps,\infty)$ such that the following holds:
\begin{enumerate}[label=(\alph*)]
\item $\tilde g^{\ell,\eps}= g^{\ell,\eps}$ on $\left((-\eps,-\delta]\cup [\delta,2+2\eps-\delta]\cup [2+2\eps+\delta,\infty)\right)\times S$;
\item The sectional curvatures of $\tilde{g}^{\ell,\eps}$ on $(-\delta,\delta)\times S$ are bounded above by $K_0$;
\item The sectional curvatures of $\tilde{g}^{\ell,\eps}$ on $(2+2\eps-\delta,2+2\eps+\delta)\times S$ are bounded above by $-(M_1-1)^2$;
\item All hypersurfaces $S_t$ are strictly convex. Moreover, the principal curvatures of $S_t$ are at least $\lambda_{\min}(S_0)/2$ for $t\in(-\delta, \delta)$. 
\end{enumerate} 
\end{proposition}

\begin{proof}
Pick a function $\psi\in C^{\infty}_c(\mathbb{R})$ such that $\psi$ is supported on $[-1,1]$, $\psi\geq 0$ and $\int_{\mathbb{R}}\psi=1$. For any $\eta>0$ define a smooth mollifier 
$$\psi_\eta(t):=\eta^{-n+1}\psi(t/\eta).$$
 For any given $\delta$, let $\beta_\delta$ be a bump function vanishing on $|t|\geq \delta$ and with value 1 for  $|t|\leq \delta/2$. We fix $\ell$ and $\eps$ and are going to smooth out $g^{\ell,\eps}$ near $\{2+2\eps\}\times S$ and near $\{0\}\times S$.

\textbf{Step 1: Smoothing near $\{2+2\eps\}\times S$:}
Notice that for $t\in [2+\eps, \infty]$, we can express $g^{\ell,\eps}_t$ in the following way: $g^{\ell,\eps}_t=\bar{f}(t)^2ds^2_{n-1}$, where
\begin{equation*}
\bar{f}(t)=\left\{\begin{aligned}
&\sqrt{f_\ell(t)}, &t\in [2+\eps, 2+2\eps];\\ 
&\frac{\sinh [\kappa(t+\tilde{r})]}{\kappa}, &t\in [2+2\eps,\infty).
\end{aligned}
\right.
\end{equation*}
Since $g^{\ell,\eps}$ is $C^{1,1}$, so is $\bar{f}$. The sectional curvature for $g$ on $t\geq 2+\eps$ is given by
$$K(\sigma)=\cos^2\theta\left(-\frac{\bar{f}''}{\bar{f}}\right)+\sin^2\theta\left(\frac{1}{\bar{f}^2}-\left(\frac{\bar{f}'}{\bar{f}}\right)^2\right),$$
where $\theta$ is the angle between the tangent 2-plane $\sigma$ and $T$.
By Lemma \ref{lemma: curvature for product}, we have
$$\bar{f}'>0, \,\,\,\frac{\bar{f}''}{\bar{f}}\geq M_1^2\,\,\,\text{ and } \,\,\,\frac{1}{\bar{f}^2}-\left(\frac{\bar{f}'}{\bar{f}}\right)^2\leq -M_1^2.$$
Take the convolution of $\bar{f}$ with $\psi_{\eta}$,
$$\bar{f}_{\eta}(t):=\int_{-\eta}^{\eta}\bar{f}(t-s)\psi_{\eta}(s)ds.$$
By properties of convolution, $\bar{f}_{\eta}\to\bar{f}$ in $C^1$ as $\eta\to 0$. Define
$$\tilde{f}_{\eta}(t):=(1-\beta_\delta(t))\bar{f}(t)+\beta_\delta(t)\bar{f}_{\eta}(t).$$
Let
$$
\tilde{g}^{\ell,\eps}_t:=\tilde{f}_{\eta}(t)^2ds^2_{n-1}.
$$
We have that $\tilde{f}_{\eta}$ is smooth and $\tilde{f}_{\eta}\to\bar{f}$ in $C^1$ topology, thus there exists $\eta_1>0$ such that for all $\eta<\eta_1$, 
$$\tilde{f}'_{\eta}>0\,\,\,\text{ and } \,\,\,\frac{1}{\tilde{f}_{\eta}^2}-\left(\frac{\tilde{f}'_{\eta}}{\tilde{f}_{\eta}}\right)^2\leq -(M_1-1)^2.$$
Hence all $S^t$ with $t\in (2+2\eps-\delta,2+2\eps+\delta)$ are strictly convex. 

In order to finish the proof of (c) we only have to estimate $\tilde{f}''_{\eta}/\tilde{f}_{\eta}$. When $|t-2-2\eps|\leq \delta/2$, $\tilde{f}_{\eta}(t)=\bar{f}_{\eta}(t)$. Thus
$$\tilde{f}''_{\eta}(t)=\int_{-\eta}^{\eta}\bar{f}''(t-s)\psi_{\eta}(s)ds
\geq \int_{-\eta}^{\eta}M_1^2\bar{f}(t-s)\psi_{\eta}(s)ds=M_1^2\tilde{f}_{\eta}(t).$$
When $|t-2-2\eps|\in [\delta/2,\delta]$, since $\bar{f}$ is $C^2$ on these intervals, we have $\tilde{f}_{\eta}\to \bar{f}$ in $C^2$ topology and we can find $\eta_2$ such that for any $\eta<\eta_2$, $\tilde{f}''_{\eta}/\tilde{f}_{\eta}\geq (M_1-1)^2$. We finish the proof by taking $\eta<\min\{\eta_1, \eta_2\}$.

\textbf{Step 2: Smoothing near $\{0\}\times S$:}
We define $\bar{g}_{\eta}:=dt^2+\bar{g}_{\eta,t}$ on $(-\delta,\delta)\times S$ via convolution
$$\bar{g}_{\eta,t}:=\int_{-\eta}^{\eta}g_{t-s}\psi_\eta(s)ds.$$
It is clear that $\bar{g}_{\eta}\to g$ in $C^1$. Since $g$ is $C^{1,1}$ with respect to $t$ and smooth with respect to coordinates on $S$,  $\frac{d^2}{dt^2}\bar{g}_{\eta}$ is bounded by the Lipchitz constant of $\frac{d}{dt}g$, while other second order derivatives of $\bar{g}_{\eta}$ converge to those of $g$. Thus all second derivatives of $\bar{g}_{\eta}$ are uniformly bounded on any compact set. Hence there exists $\eta_3>0, K_0>K_g$ such that for any $\eta\in (0,\eta_3)$, the sectional curvatures of  $\bar{g}_{\eta}$ are bounded above by $K_0/2$ on $[-\delta,\delta]\times S$. 

Define 
$$\tilde{g}^{\ell,\eps}_{\eta}:=dt^2+\tilde{g}^{\ell,\eps}_{\eta,t},$$
 where
$$\tilde{g}^{\ell,\eps}_{\eta,t}:=(1-\beta_\delta(t))g^{\ell,\eps}_t+\beta_\delta(t)\bar{g}^{}_{\eta,t}.$$

We need to establish the bounds on sectional curvature when $|t|\in [\delta/2,\delta]$. Notice that in these domains $g$ is at least $C^2$, thus $\bar{g}^{}_{\eta}\to g$ in $C^2$ as $\eta\to 0$ on both $[\delta/2, \delta]\times S$ and $ [-\delta, -\delta/2]\times S$.
Hence for any fixed $\delta$, $\tilde{g}^{\ell,\eps}_{\eta}\to g$ in $C^2$ topology on these domains. Since $K_0>K_g$ and the curvature of $g^{\ell,\eps}$ on $(-\delta, \delta)\times S$ is bounded above by $K_g$ by Proposition \ref{prop: C11 extension}(c), there exists $\eta_4>0$ such that for any $\eta<\eta_4$, the sectional curvatures of $\tilde{g}^{\ell,\eps}_{\eta}$ on both $[\delta/2, \delta]\times S$ and $ [-\delta, -\delta/2]\times S$ are bounded from above by $K_0$. Thus we obtain item (b).

Now we prove (d), since $\tilde{g}^{\ell,\eps}_{\eta}\to g$ in $C^1$ topology as $\eta\to 0$ and principal curvatures depend merely on $\tilde{g}^{\ell,\eps}_{\eta}$ and $\frac{d}{dt}\tilde{g}^{\ell,\eps}_{\eta}$, by Proposition \ref{prop: C11 extension}(b) and the assumption on $\eps$ above Proposition \ref{prop: C11 extension}, we know that there exists $\eta_5>0$ such that for $\eta<\eta_5$, the principal curvatures has a uniform lower bound $\lambda_{\min}(S_0)/2$. 

We finish the proof by taking $\tilde{g}^{\ell,\eps}:=\tilde{g}^{\ell,\eps}_{\eta}$ with $0<\eta<\min\{\eta_3, \eta_4, \eta_5\}$.
\end{proof}

\section{Anosov extension}\label{sec: anosov_ext}

The goal of this section is to prove the main theorem whose statement we recall. 

\begin{theorem}[Theorem A]\label{thm_anosov_ext}
 Let $(\Sigma, g)$ be a compact smooth Riemannian manifold with boundary. Assume that each component of the boundary is a strictly convex sphere. Also assume that $(\Sigma, g)$ has no conjugate points and the trapped set for the geodesic flow is hyperbolic. Then, there exists a codimension 0 isometric  embedding $(\Sigma,g)\subset (\Sigma^{ext}, g^{ext})$ such that $(\Sigma^{ext}, g^{ext})$ is a closed Anosov manifold.
\end{theorem}

We first describe the main construction where we allow $\partial\Sigma$ to have several connected components. Afterwards, we need to establish the estimates on Jacobi fields, which then allow us to prove the absence of conjugate points and to finish the proof in Section~\ref{section: absence of conjugate points and the main theorem}. For the sake of simpler notation, in this part of the proof we assume that $\partial\Sigma$ has only one connected component. The argument for the general case is the same.

\subsection{Description of the extension}\label{subsec: description}

To describe the extension, we will need the following fact.

\begin{lemma}[\cite{G17}, Lemma 2.3] \label{lem: extend_metric}
For any sufficiently small $\delta_0>0$, there exists an isometrical embedding of $(\Sigma, g)$ into a smooth Riemannian manifold $(\Sigma^{\delta_0}, g^{\delta_0})$ with strictly convex boundary which is equidistant to the boundary of $\Sigma$, has the same hyperbolic trapped set as $(\Sigma,g)$, and no conjugate points. Moreover, all hypersurfaces equidistant to the boundary of $\Sigma$ in $\Sigma^{\delta_0}\setminus \Sigma$ are strictly convex. 
\end{lemma}

By the lemma we can fix a $\delta_0>0$ such that the principal curvatures of all hypersurfaces equidistant to the boundary of $\Sigma$ in $\Sigma^{\delta_0}\setminus \Sigma$ are at least $\lambda_{\min}(\partial\Sigma)/2$ where $\lambda_{\min}(\partial\Sigma)$ is the minimum of principal curvatures of $\partial\Sigma$. 

We denote by $Q_0:=Q_{\Sigma^{\delta_0}}$ and $C_0:=C_{\Sigma^{\delta_0}}$ the constants given by Proposition \ref{prop: highent_highexit} when applied to $\Sigma^{\delta_0}$. 
Assume $\partial\Sigma^{\delta_0}=\sqcup_{j=1}^m S^j$ with each $S^j$ diffeomorphic to a sphere. For any sufficiently small { $\eps\in(0,\delta_0)$,} we can consider normal coordinates in the $\eps$-neighborhoods of each $S^j$. In particular, for each $j$, the $\eps$-neighborhood of $S^j$ is isometric to $(-\eps,0]\times S^j$ with metric ${}_jg=dt^2+{}_jg_t$ where $t\in(-\eps,0]$ parametrizes the (signed) distance to $S^j$ and ${}_jg_t$ is the Riemannian metric on $\left(S^j\right)_t=\{t\}\times S^j$. Recall that a metric in Proposition~\ref{prop: smoothify} is the smoothing of a metric in Proposition~\ref{prop: C11 extension}. By applying Proposition~\ref{prop: smoothify}, for any $M_1>1$, $\delta\in(0,\frac{\eps}{2})$ and
\begin{equation}\label{cond: collar_curv}
\ell>\max_{j}\{L(M_1^2,M_1, {}_jg, \eps, \rho)\} \,\,\,\,\,\,\text{(see Proposition~\ref{prop: C11 extension} for the definition of $L$),}
\end{equation}
there exists a smooth Riemannian metric ${}_j\tilde g^{\ell,\eps}$ on $(-\eps,\infty)\times S_j$  for each $j$ with the properties listed in Proposition~\ref{prop: smoothify}. Let $\kappa$ and $\tilde r$ be as in Proposition~\ref{prop: C11 extension} for the chosen $M_1, M_0=M_1^2$ and $\ell$. Then, we excise $\eps$-neighborhood of the boundary of $\Sigma^{\delta_0}$ and replace $(-\eps,0]\times S$ with metric $g^{\delta_0}$ with $\sqcup_{j=1}^m(-\eps, 2+2\eps+\delta]\times S_j$ where each $(-\eps, 2+2\eps+\delta]\times S_j$ is equipped with the metric ${}_j\tilde g^{\ell,\eps}$. We denote the resulting Riemannian manifold with constant curvature $-\kappa^2$ near the boundary by $(\Sigma^{\delta_0}_{\ell,\eps}, g^{\delta_0}_{\ell,\eps})$. Notice that, since $\delta\in(0,\frac{\eps}{2})$, the manifold $(\Sigma^{\delta_0}_{\ell,\eps}, g^{\delta_0}_{\ell,\eps})$ contains an isometric copy $(\Sigma,g)$.

Fix $R>0$. By Proposition~\ref{prop: C11 extension} each metric ${}_j\tilde g^{\ell,\eps}$ has the form $\left(\frac{1}{\kappa}\sinh[\kappa(t+\tilde r)]\right)^2ds^2_{n-1}$ for $t > 3$, which is the form of the hyperbolic metric constant curvature $-\kappa^2$ on $\mathbb H^n$. Therefore we can remove $m$ balls from $\mathbb H^n$ and replace them with $(\Sigma^{\delta_0}_{\ell,\eps}, g^{\delta_0}_{\ell,\eps})$ in such a way that the distance between different components is at least $R$. Clearly we can also perform the same surgery procedure starting from a closed hyperbolic manifold of curvature $-\kappa^2$ provided that the injectivity radius is sufficiently large. 
Existence of such hyperbolic manifolds is well-known and follows from the residual finiteness of the fundamental groups of hyperbolic manifolds. We include the proof for the sake of completeness.
\begin{lemma}\label{lemma: large injectivity radius}
Let $M$ be a compact hyperbolic manifold. Given any $D>0$ there exists a finite cover $\tilde M\to M$ such that the injectivity radius of $\tilde M$ is $\ge D$.
\end{lemma}

\begin{proof}
Let $\alpha_1, ... \alpha_N$ be the list of closed geodesics on $M$ whose length is less than $2D$ and let $\beta_1, ... \beta_N$ be the elements of $\pi_1(M,x_0)$ which are freely homotopic to these geodesics. Because $\pi_1(M,x_0)$ is residually finite~\cite{Mal} there exists a finite group $G$ and a homomorphism $h\colon \pi_1(M,x_0)\to G$ such that $h(\beta_i)\neq id_G$. Then the finite cover $\tilde M$ which corresponds to kernel of $h$ has injectivity radius $>D$.
\end{proof}

Thus we  obtain
 a smooth closed Riemannian manifold $(\Sigma^{ext}, g^{ext})$ which contains an isometric copy of $(\Sigma,g)$. To guarantee that the constructed extension is Anosov $(\Sigma^{ext}, g^{ext})$ , we make some choice of parameters $\eps, \ell, \delta, M_1,$ and $R$ such that they  satisfy the following conditions:

\begin{enumerate}
\item [(C1)] $M_1$ is sufficiently large;
\item [(C2)] \label{cond: eps} $\eps$ is sufficiently small;
\item [(C3)]\label{cond: delta} $\delta<\min\{\delta_0, \eps/2\}$ and is sufficiently small;
\item [(C4)] $R$  is sufficiently large.
\end{enumerate}
The precise conditions of  above constants can be found in Appendix \ref{app: const_in_M_ext}.

\begin{remark}\label{remark about kappa}
    We want to point out that the resulting constant sectional curvature $\kappa$ in the extension can be a priori arbitrarily large, and its value depends on $\ell$ which depends on the given Riemannian manifold with boundary . This can be seen from Lemma~\ref{lemma: combine f with constant 2}.
\end{remark}

We introduce notation that we will use in the next sections.

Denote by $\mathcal{C}_+^1:=\cup_{j=1}^m[-\delta, \delta] \times S^j $ and $\mathcal{C}_+^2:=\cup_{j=1}^m[\delta, \eps] \times S^j$.  We decompose $\Sigma^{ext}$ into three domains
$$\Sigma^{ext}=\Sigma_0\cup \mathcal{C}_+\cup \mathcal{D}_-,$$
where $\Sigma_0:=\Sigma \cup\cup_{j=1}^m [-\delta_0, -\delta]\times S^j, \mathcal{C}_+:=\mathcal{C}_+^1\cup \mathcal{C}_+^2$ and $\mathcal{D}_-:=\Sigma^{ext}\setminus (\Sigma_0\cup \mathcal{C}_+)$. 

We summarize the properties of the resulting extension that come from Propositions \ref{prop: highent_highexit}, Propositions~\ref{prop: C11 extension} and \ref{prop: smoothify} with our choice of parameters:

(i) We have the conclusion of Proposition \ref{prop: highent_highexit} for $(\Sigma_0, g^{ext})$ with $Q_0$ and $C_0$. 

(ii) The sectional curvatures on $\mathcal{D}_-$ are at most $-(Q_0+3)^2$. And all maximal geodesic segments within $\mathcal{D}_-$ have length at least $R$.

(iii) On $\mathcal{C}_+^1$, the curvature upper bound is $K_0$ and the principal curvatures for hypersurfaces in $\mathcal{C}_+^1$ equidistant to $\Sigma$ are at least $\lambda_{\min}(\partial\Sigma)/4$. 

(iv) On $\mathcal{C}_+^2$, the curvature upper bound is $K_g$ and the principal curvatures for hypersurfaces in $\mathcal{C}_+^2$ equidistant to $\Sigma$ are at least $\lambda_{\min}(\partial\Sigma)/2$.

\subsection{Travel time and Jacobi estimate in the collar}

As we mentioned before, for the sake of simpler notation, we assume that $\partial\Sigma$ has only one connected component.

We denote the  boundary of $\Sigma^{\delta_0}$ by $S$ (see Section~\ref{subsec: description}) and let $S_t=\{t\}\times S$.

We want to estimate the travel time and change of $\mu_J$ when a geodesic goes through $\mathcal{C}_+$. { To do that we consider} a setting which is (formally)  more general than $(iii)$ and $(iv)$ above which we proceed to describe. 

Let $c: [0,\tau]\to [b_-, b_+] \times S $ be a unit speed geodesic segment in $[b_-, b_+] \times S$ on which the sectional curvature is bounded from above by $\kappa_0>0$. We may assume the principal curvatures of $S_t\,(b_-\leq t\leq b_+)$ are at least $\lambda>0$. Namely, the shape operator satisfies
\begin{equation}\label{eq: lower_bd_A}
\langle A(t)X,X\rangle\geq \lambda \|X\|^2,\,\,\,\,\forall t\in[b_-, b_+], X\in S_t.
\end{equation}
Moreover, we assume that
\begin{equation}\label{eq: lower_bd_width}
b_+-b_-<\frac{1}{\lambda}\ln\cosh\frac{\lambda}{2\kappa_0+2(Q+1)^2}
\end{equation}
for some $Q>0$. 

For any  $s\in[0,\tau]$, let $d(s)$ be the $t$-coordinate of $c(s)$. By the first variation formula, $d'(s)=\langle T, \dot{c}(s)\rangle$. Let $W(s)$ be the component of $\dot{c}(s)$ orthogonal to $T$. Then, we have $\|W(s)\|^2=1-d'(s)^2, \nabla_{\dot{c}(s)} T=\nabla_{W(s)} T\text{ and }\nabla_{W(s)} T\perp T.$ Hence, by the second variation formula, 
\begin{equation}\label{eq: d''}
d''(s)=\langle \nabla_ {\dot{c}(s)} T,  \dot{c}(s)\rangle=\langle \nabla_ {W(s)} T,  W(s)\rangle=\langle A(d(s))W(s), W(s)\rangle.
\end{equation}
 
\begin{lemma}\label{lem: bdd_trav_time}
The travel time in the collar has the following upper bound
$$\tau\leq(\kappa_0+(Q+1)^2)^{-1}.$$
 For any perpendicular Jacobi field $J$ along $c$ with $J(0)\neq 0$, if $\mu_J(0)\geq -Q$, then $\mu_J(t)>-Q-1$ for $t\in [0,\tau]$ and 
 $$\int_0^\tau \mu_J(t)dt\geq -\frac{1}{Q+1}.$$
  Similarly, if $\mu_J(0)>Q+1$, then $\mu_J(t)>Q$ for $t\in [0,\tau]$.
 \end{lemma}
\begin{proof}
If $|d'(s_0)|=1$ for some $s_0\in [b_-, b_+]$ then $\dot c(s_0)=T$ and therefore $\dot{c}(s)=T$ for all $s\in [b_-, b_+]$ thus the travel time is $\tau=b_+-b_-$. Hence we can assume $|d'(s)|<1$ for all $s\in [0,\tau]$. By (\ref{eq: lower_bd_A}) and (\ref{eq: d''}), we have 
$$d''=\langle \nabla_W T, W\rangle=\|W\|^2\left\langle A(d(s))\left(\frac{W}{\|W\|}\right), \frac{W}{\|W\|}\right\rangle\geq (1-(d')^2)\lambda.$$
Assume $d(t_0)=\min d(s)$. If $t_0\in(0,\tau)$ then $d'(t_0)=0$, while $t_0=0$ implies that $d'(t_0)>0$. The case when $t_0=\tau$ is symmetric to $t_0=0$. Thus we may assume $d'(t_0)\geq 0$. For $s\geq t_0$ we have
$$\frac{1}{2}\ln\left|\frac{1+d'(s)}{1-d'(s)}\right|=\frac{1}{2}\ln\left|\frac{1+d'(t_0)}{1-d'(t_0)}\right|+\int_{t_0}^s \frac{d''(\tau)}{1-d'(\tau)^2}d\tau\geq\lambda (s-t_0),$$
which implies that $d'(s)\geq \tanh(\lambda (s-t_0))$. Hence, 
$$d(s)=d(t_0)+\int_{t_0}^sd'(\tau)d\tau \geq b_-+\int_{0}^{s-t_0}\tanh(\lambda\tau)d\tau=b_-+\frac{\ln\cosh(\lambda (s-t_0))}{\lambda}.$$
On the other hand, $d(s)\leq b_+$ for all $s\in [0,\tau]$. Together with \eqref{eq: lower_bd_width} we obtain
$$\tau-t_0\leq \frac{1}{\lambda}\cosh^{-1}e^{\lambda(b_+-b_-)}<\frac{1}{2\kappa_0+2(Q+1)^2}.$$
Thus, again by symmetry, we have 
$$\tau\leq(\kappa_0+(Q+1)^2)^{-1}.$$ 
 Now we estimate the change of $\mu_J$. The solution of $u''+\kappa_0u=0$ with $u(0)=1, u'(0)=-Q$ satisfies 
$$\frac{u'(t)}{u(t)}=-\sqrt{\kappa_0}\tan\left(\sqrt{\kappa_0}t+\tan^{-1}\frac{Q}{\sqrt{\kappa_0}}\right), \,\,\,\,\, t\in [0,\tau].$$
By Mean Value Theorem, 
$$\tau\leq \frac{1}{\kappa_0+(Q+1)^2}<\frac{1}{\sqrt{\kappa_0}}\tan^{-1}\frac{Q+1}{\sqrt{\kappa_0}}-\frac{1}{\sqrt{\kappa_0}}\tan^{-1}\frac{Q}{\sqrt{\kappa_0}}$$
Thus $u'(t)/u(t)>-Q-1$ for $0\leq t\leq \tau$. Since the sectional curvature in $[b_-, b_+] \times S$ is bounded from above by $\kappa_0$, applying Lemma \ref{lem: comp_thm} with $f\equiv \kappa_0$ on $[0,\tau]$, we obtain $\mu_J(t)\geq u'(t)/u(t)>-Q-1$. Thus 
$$\int_0^\tau \mu_J(t)dt\geq -\frac{Q+1}{\kappa_0+(Q+1)^2}\geq -\frac{1}{Q+1}.$$
The last assertion of the lemma follows by using the argument by contradiction and reversing time.
\end{proof}

\begin{corollary}\label{cor: inf_at_C1}
Let $J$ be a nonzero perpendicular Jacobi field along $c$ with $J(t^*)=0$ for some $t^*\in(0,\tau)$, then $J$ does not vanish on $(t^*,\tau]$ and $\mu_J(\tau)>Q$.
\end{corollary}

\begin{corollary}\label{cor: mu_change_C+}
Let $c:[0,\tau_0]\to \mathcal{C}_+$ be a geodesic in $\mathcal{C}_+$ and $J$ be a perpendicular Jacobi field along $c$. 

(a) If $c(0)\in S_{-\delta}, c(\tau_0)\in S_{\eps}$ and $\mu_J(0)>-Q_0$, then $\mu_J(t)>-Q_0-2$ for all $t\in [0,\tau_0]$ and 
$$\int_0^{\tau_0} \mu_J(t)dt> -\frac{2}{Q_0+1}.$$

(b) If $c(0)\in S_{\eps}, c(\tau_0)\in S_{-\delta}$ and $\mu_J(0)>Q_0+2$, then $\mu_J(t)>Q_0$ for all $t\in [0,\tau_0]$.

(c)  If both $c(0), c(\tau_0)\in S_{\eps}$ and $\mu_J(0)>Q_0+2$, then $\mu_J(t)>-Q_0-2$ for all $t\in [0,\tau_0]$ and 
$$\int_0^{\tau_0} \mu_J(t)dt> -\frac{2}{Q_0+1}.$$
\end{corollary}
\begin{proof}
On $\mathcal{C}_+^1$ (resp. $\mathcal{C}_+^2$), we apply Lemma \ref{lem: bdd_trav_time}  with $\kappa_0=K_0$ (resp. $\kappa_0=K_g$), $\lambda=\lambda_{\min}(\partial\Sigma)/4$ (resp. $\lambda=\lambda_{\min}(\partial\Sigma)/2$) and (C3) (resp. (C2)) is equivalent to condition \eqref{eq: lower_bd_width}  with $Q=Q_0$ (resp. $Q=Q_0+1$). 

(a) Since all hypersurfaces $S_t$ are convex, there exists $\tau\in [0,\tau_0]$ such that $c[0,\tau]\subseteq \mathcal{C}_+^1$ and $c[\tau, \tau_0]\subseteq \mathcal{C}_+^2$. By applying Lemma \ref{lem: bdd_trav_time} on both $\mathcal{C}_+^1$ and $\mathcal{C}_+^2$ we obtain $\mu_J(t)>-Q_0-1$ on $[0,\tau]$ and $\mu_J(t)>-Q_0-2$ on $[\tau, \tau_0]$. Thus
$$\int_0^{\tau_0} \mu_J(t)dt\geq -\frac{1}{Q_0+1}-\frac{1}{Q_0+2}> -\frac{2}{Q_0+1}.$$

(b) This item follows by reversing time and applying (a).

(c) If $c[0,\tau_0]$ does not intersects $\mathcal{C}_+^1$, then applying Lemma \ref{lem: bdd_trav_time}  on $\mathcal{C}_+^2$ implies $\mu_J(t)>Q_0+1$ for all $t\in [0,\tau_0]$. Otherwise assume $c[a, b]\subseteq \mathcal{C}_+^1$, then we get (c) by applying Lemma \ref{lem: bdd_trav_time} three times on $c[0,a], c[a,b]$ and $c[b,\tau_0]$. The estimate on the integral follows from item (a).
\end{proof}

\subsection{Jacobi field estimate outside $\Sigma_0$}

The following lemma allows us to estimate how Jacobi fields change outside $\Sigma_0$.
\begin{lemma}\label{lem: mu_change_NSigma}
Let $c: [\tau_1,\tau_2]\to \mathcal{C}_+\cup\mathcal{D}_-$ be a maximal geodesic with $c(\tau_1)\in \partial\Sigma_0\cup  \mathcal{D}_-$,  and
 $J$  a perpendicular Jacobi field along $c$ with $-Q_0<\mu_J(\tau_1)<\infty$, then $J(t)\neq 0$ for all $t\in[\tau_1,\tau_2]$. Moreover,
\begin{enumerate}[label=(\roman*)]
\item\label{tau2_finite} If $\tau_2<\infty$,  and $c(\tau_1)\in \partial\Sigma_0$, then $\mu_J(\tau_2)>Q_0$ and $\int_{\tau_1}^{\tau_2}\mu_J(t)dt>Q_0+C_0+2$;

\item\label{tau2_infinite} If $\tau_2=\infty$, then $\int_{\tau_1}^{\tau_2}\mu_J(t)dt=\infty$;

\item\label{boundontau12} $\mu_J(t)>-Q_0-2$ for all $t\in[\tau_1,\tau_2]$.
\end{enumerate}
\end{lemma}
\begin{proof}

Let $\{a_i\}^n_{i=0}$ and $\{b_i\}^n_{i=0}$ be the sequences of times with $\tau_1\leq a_0<b_0<\ldots a_n<b_n<a_{n+1}\leq \tau_2$ ($n$ and $\tau_2$ could be $\infty$) such that
$c[a_k, b_k] (k=0,1,\cdots,n)$ are the geodesic segments in $\mathcal{C}_+$, and $c(\tau_1,a_0)$, $c(b_k,a_{k+1})  (k=0,1,\cdots,n)$, and $c(a_{n+1},\tau_2)$ are contained in $\mathcal D_{-}$.  Since $c(\tau_1)\in \partial\Sigma_0\cup  \mathcal{D}_-$ and $-Q_0<\mu_J(\tau_1)<\infty$,   we have $\mu_J(a_0)\in (-Q_0, \infty)$.

By construction of $\Sigma^{ext}$, we know that for all $0\leq k\leq n,$
$$a_{k+1}-b_k>R.$$

Firstly, we prove that 

\begin{align}
\mu_J(b_k)>-Q_0-2\quad \Rightarrow \quad & \mu_J(t)>-Q_0-2 \text{ for }t\in[b_k,a_{k+1}],  \nonumber\\
&  \mu_J(a_{k+1})>Q_0+2\,\,\,\,\text{ and  } \label{eq: lower_bdd_bkak+1}\\
&  \int_{b_k}^{a_{k+1}}\mu_J(t)dt>Q_0+C_0+2+\frac{2}{Q_0+1}.\label{eq: lower_bdd_int_bkak+1}
\end{align}
Indeed, $c[b_k, a_{k+1}]\subseteq \mathcal{D}_-$ on which the sectional curvatures are bounded above by $-(Q_0+3)^2$. By Lemma \ref{lem: comp_thm}, we know that for $t\in[0, a_{k+1}-b_k]$,
$$\mu_J(b_k+t)>(Q_0+3)\tanh\left(t(Q_0+3)-\tanh^{-1}\frac{Q_0+2}{Q_0+3}\right)$$
By choosing $R$ sufficiently large,  $\mu_J(b_k+t)>-Q_0-2$ for $t\in[0,a_{k+1}-b_k]$. Moreover, we have $\mu_J(b_k+t)>Q_0+2$ for all $t\in \left[\frac{2}{Q_0+3}\tanh^{-1}\frac{Q_0+2}{Q_0+3}, a_{k+1}-b_k\right]$ and
\begin{equation}\label{eq: lower_bdd_int_bkbk+t}
\int_{b_k}^{b_k+t}\mu_J(\tau)d\tau>(Q_0+2)\left(t-\frac{2}{Q_0+3}\tanh^{-1}\frac{Q_0+2}{Q_0+3}\right).
\end{equation}
In particular, we have \eqref{eq: lower_bdd_bkak+1} and \eqref{eq: lower_bdd_int_bkak+1}. Together with Corollary \ref{cor: mu_change_C+}, we have the following two statements:
\begin{eqnarray}\label{eq: lower_bdd_akak+1}
\mu_J(a_k)>Q_0+2&\Rightarrow &\mu_J(t)>-Q_0-2 \text{ for }t\in[a_k,a_{k+1}], \\
& & \mu_J(a_{k+1})>Q_0+2 \text{ and  }\nonumber\\
& & \int_{a_k}^{a_{k+1}}\mu_J(t)dt>Q_0+C_0+2.
\end{eqnarray}
\begin{eqnarray}\label{eq_lower_bdd_bkbk+1}
\mu_J(b_k)>-Q_0-2 &\Rightarrow & \mu_J(t)>-Q_0-2 \text{ for }t\in[b_k,b_{k+1}], \\
& & \mu_J(b_{k+1})>-Q_0-2 \text{ and  }\nonumber\\
& & \int_{b_k}^{b_{k+1}}\mu_J(t)dt>Q_0+C_0+2.
\end{eqnarray}

\begin{figure}[t]
\centering
\includegraphics[scale=0.75]{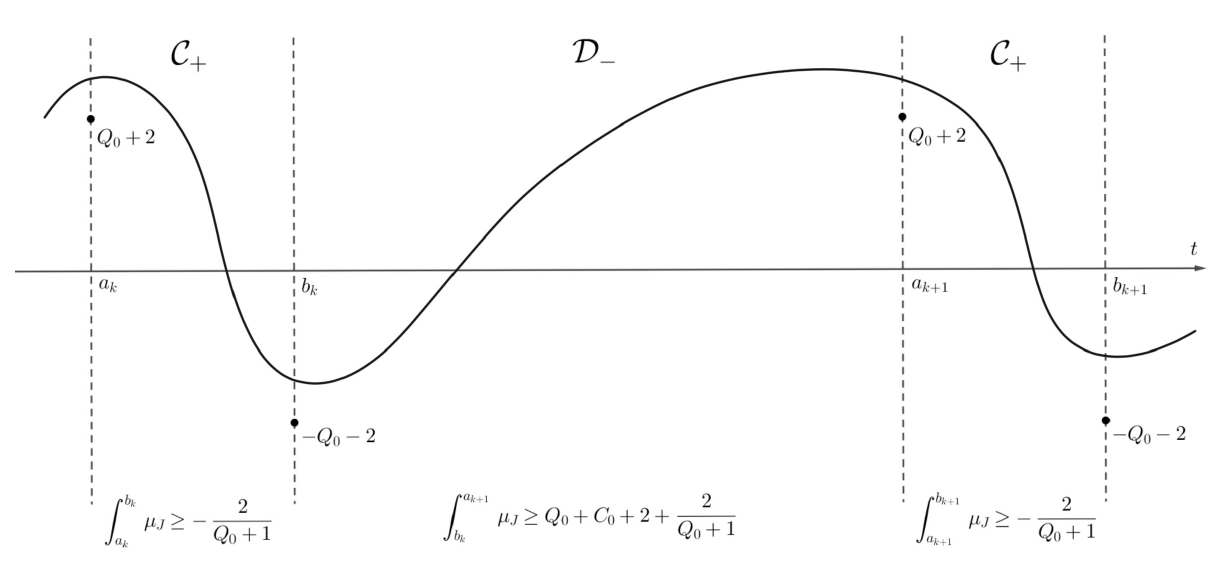}
\caption{Graph of $\mu_J$}
\label{fig: graph_of_mu}
\end{figure}

Now we make the estimate on the entire $[\tau_1,\tau_2]$. Since $\mu_J(\tau_1)>-Q_0$,  by Corollary \ref{cor: mu_change_C+}(a), we know that $\mu_J(b_0)\geq -Q_0-2$. By \eqref{eq: lower_bdd_bkak+1} and \eqref{eq: lower_bdd_akak+1}, we obtain \ref{boundontau12} and for any $k\geq 1$, $\mu_J(a_k)>Q_0+2$ and
$$\int_{\tau_1}^{\tau_2}\mu_J(\tau)d\tau>\int_{a_0}^{a_n}\mu_J(\tau)d\tau=\sum_{k=0}^{n-1}\int_{a_k}^{a_{k+1}}\mu_J(\tau)d\tau>n(Q_0+C_0+2).$$
Thus, when $n=\infty$, $\int_{\tau_1}^{\tau_2}\mu_J(\tau)d\tau=\infty$. When $\tau_2<\infty$, by \eqref{eq: lower_bdd_bkak+1} and \eqref{eq: lower_bdd_akak+1},  $\mu_J(a_n)>Q_0+2$, thus $\mu_J(\tau_2)>Q_0$ due to Corollary \ref{cor: mu_change_C+} and we obtain \ref{tau2_finite}. The only case left is when $n<\infty$ but $\tau_2=\infty$. In this case we apply (\ref{eq: lower_bdd_int_bkbk+t}) and obtain
$$\int_{\tau_1}^{\tau_2}\mu_J(\tau)d\tau=\int_{\tau_1}^{b_n}\mu_J(\tau)d\tau+\int_{b_n}^{\infty}\mu_J(\tau)d\tau=\infty.$$
\end{proof}

\begin{corollary}\label{cor: inf_at_NSigma}
Let $c: [\tau_1,\tau_2]\to \mathcal{C}_+\cup\mathcal{D}_-$ be a geodesic segment and $J$ be a nonzero perpendicular Jacobi field along $c$ with $J(t^*)=0$ for some $t^*\in(\tau_1,\tau_2)$, then $\mu_J>-Q_0-2$ on $(t^*,\tau_2]$. In particular, $J$ does not vanish on $(t^*,\tau_2]$.
\end{corollary}
\begin{proof}
If $t^*\in (a_k,b_k)$ for some $k$, then $\mu_J>Q_0$ on $(t^*,b_k]$ via Corollary \ref{cor: inf_at_C1}. Thus, $\mu_J>-Q_0-2$ on $(t^*,\tau_2]$ follows from Lemma \ref{lem: mu_change_NSigma}\ref{boundontau12}. If $t^*\in (b_k, a_{k+1})$ for some $k$, apply Lemma \ref{lem: comp_thm}  we have $\mu_J>Q_0+2$ on $(t^*, a_{k+1}]$. Finally, from Lemma \ref{lem: mu_change_NSigma}\ref{boundontau12}, we have $\mu_J>-Q_0-2$ on $(t^*,\tau_2]$. 
\end{proof}

\subsection{Proof of absence of conjugate points and of the main theorem}\label{section: absence of conjugate points and the main theorem}
In order to prove $(\Sigma^{ext}, g^{ext})$ is Anosov, we first prove the absense of conjugate points.
\begin{proposition}\label{prop: noconjpts}
The extension $(\Sigma^{ext}, g^{ext})$ has no conjugate points.
\end{proposition}
\begin{proof}
We need to prove that for any geodesic $\gamma_v$ and perpendicular Jacobi field $J$ along $\gamma_v$, if $J(t^*)=0$, then $J(t)\neq 0$ for all $t>t^*$. Assume $t^-_1\leq t^+_1<t^-_2\leq t^+_2<\cdots$ are the times when $\gamma_v$ crosses $\partial\Sigma_0$ and we assume that $\gamma_v[t^-_k, t^+_k], k\in\mathbb{Z}$ are the segments within $\Sigma_0$. 

\begin{lemma}\label{lem: t*<t+n}
For any $n$ with $t^+_n>t^*$, we have $\mu_J(t^+_{n})>-Q_0$ and $J$ does not vanish on $(t^*, t^+_n]$.
\end{lemma}
\begin{proof}
Firstly, we prove the statement for the first $n$ with $t^+_n>t^*$.  If $t^*\in [t^-_k, t^+_k]$ for some $k$ then $\mu_J(t^+_k)>-Q_0$, otherwise by reversing time we obtain a Jacobi field, $J^*$, entering $\Sigma_0$ with $\mu_{J^*}>Q_0$ but vanishing within $\Sigma_0$, contradicting Proposition \ref{prop: highent_highexit}. If $t^*\in (-\infty, t^-_1]$ then $\mu_J(t^-_{1})>Q_0$ via Corollary \ref{cor: inf_at_C1}. Thus, $\mu_J(t^+_{1})>-Q_0$ by Proposition \ref{prop: highent_highexit}. Similar argument can be applied  when $t^*\in [t^+_k, t^-_{k+1}]$ to obtain $\mu_J(t^+_{k+1})>-Q_0$.

For general $n$, notice that $\mu_J(t^+_{n})>-Q_0$ implies $\mu_J(t^+_{n+1})>-Q_0$ due to Lemma \ref{lem: mu_change_NSigma}\ref{tau2_finite} and Proposition \ref{prop: highent_highexit}.
\end{proof}

We finish the proof of the proposition by considering the cases for the sequence of times $\{t_i^{\pm}\}$.

{\bf Case 1:} The sequence $\{t_i^{\pm}\}$ is empty. This means that $\gamma_v$ never enters $\Sigma_0$. Then the non-vanishing property of $J$ follows from Corollary \ref{cor: inf_at_NSigma}.

{\bf Case 2:} The sequence $\{t_i^{\pm}\}$ never ends. In this case $J$ does not vanish for $t>t^*$ due to Lemma \ref{lem: t*<t+n}.

{\bf Case 3:} The sequence $\{t_i^{\pm}\}$ ends with some $t^+_m$. If $t^*<t^+_m$ then by Lemma \ref{lem: t*<t+n} we have $\mu_J(t^+_{m})>-Q_0$ and Lemma \ref{lem: mu_change_NSigma} tells us that $J$ does not vanish after $t^+_m$.  If $t^*\geq t^+_m$ then Lemma \ref{lem: mu_change_NSigma} can be applied again to show that $J$ does not vanish for $t>t^*$. 

{\bf Case 4:} The sequence $\{t_i^{\pm}\}$ ends with some $t^-_m$. In this case $\gamma_v$ ends up in $\Gamma_-$ at time $t^-_m$. If $t^*<t^+_{m-1}$, then $\mu_J(t^+_{m-1})>-Q_0$ by Lemma \ref{lem: t*<t+n} and thus $\mu_J(t^-_{m})>Q_0$ by Lemma \ref{lem: mu_change_NSigma}. Therefore $J$ does not vanish after $t^-_{m}$ due to Proposition \ref{prop: highent_highexit}. If $t^+_{m-1}\leq t^*<t^-_{m}$, then we again have $\mu_J(t^-_{m})>Q_0$ by Corollary \ref{cor: inf_at_NSigma}. If $t^*\geq t^-_{m}$, then $J$ does not vanish after $t^*$ since $\Sigma_0$ has no conjugate points.
\end{proof}

Now we are ready to prove the geodesic flow on $(\Sigma^{ext}, g^{ext})$ is Anosov.

\begin{proof}[Proof of Theorem \ref{thm_anosov_ext}]

By Theorem \ref{EberleinAnosov} and Proposition \ref{prop: noconjpts}, in order to show the geodesic flow is Anosov, it suffices to prove that all non-zero perpendicular Jacobi fields on a manifold without conjugate points are unbounded. 

If a geodesic $\gamma_v$ stays in $\Sigma_0$ for all $t\in\mathbb{R}$, then $v\in\Lambda$. Thus any Jacobi field along $\gamma_v$ is unbounded by hyperbolicity. Therefore it remains to consider the case when $\gamma_v$ passes through $\mathcal{D}_-$. Let $J$ be a Jacobi field along $\gamma_v$. By changing the starting time we may assume that the geodesic segment $\gamma_v|_{[-R/2,R/2]}$ lies within $\mathcal{D}_-$. We can also assume that $J(0)\neq 0$ and $\mu_J(0)\geq 0$ (otherwise we can replace $v$ with $-v$). We will show that $\|J\|(t)\to\infty$ as $t\to\infty$. 

Recall that $\mu_J=\|J\|'/\|J\|$, hence we have only to prove the integral of $\mu_J$ is unbounded on $[0,+\infty)$.
As before denote by $0<t^-_1\leq t^+_1<t^-_2\leq t^+_2<\cdots$ the moments $\gamma_v$ crosses $\partial\Sigma_0$ with $\gamma_v[t^-_k, t^+_k], k\in\mathbb{Z}$ being the segments within $\Sigma_0$. 

{\bf Case 1:} Geodesic $\gamma_v$ never enters $\Sigma_0$ on $t\geq 0$. We decompose $\gamma_v[0,+\infty)$ using $0<a_1<b_1<a_2<b_2<\cdots$ as in the proof of Lemma \ref{lem: mu_change_NSigma}. If $a_1=+\infty$, by Lemma \ref{lem: comp_thm} we know that $\|J\|$ is unbounded. Now we assume $a_1,b_1<+\infty$, by  Lemma \ref{lem: comp_thm} again we have $\mu_J(a_1)>Q_0+2$ thus $\mu_J(b_1)>-Q_0-2$ by Corollary \ref{cor: mu_change_C+}. The unboundedness of $\|J\|$ is a consequence of Lemma \ref{lem: mu_change_NSigma}(ii).

{\bf Case 2:} Geodesic $\gamma_v$ enters $\Sigma_0$ infinitely many times on $t\geq 0$. Since $t^-_1=b_l$ for some $l\geq 1$, the argument as in Case 1 can be carried out to obtain $\mu_J(t^-_1)>Q_0$. Then we proceed by induction to get $\mu_J(t^-_k)>Q_0$ and $\mu_J(t^+_k)>-Q_0$ for all $k\geq 1$. Moreover  Proposition \ref{prop: highent_highexit} implies that 
$$\int_{t^-_k}^{t^+_k}\mu_J(t)dt\geq -C_0.$$
 For each $k$, by Lemma \ref{lem: mu_change_NSigma}(i) we have 
 $$\int_{t^+_k}^{t^-_{k+1}}\mu_J(t)dt\geq Q_0+C_0+2$$
   hence 
   $$\int_{t^-_k}^{t^-_{k+1}}\mu_J(t)dt\geq Q_0+2.$$ Thus the integral of $\mu_J$ is unbounded. 

{\bf Case 3:} The sequence $\{t_i^{\pm}\}$ ends with some $t^+_m$. The argument in Case 2 implies $\mu_J(t^+_m)>-Q_0$. The norm $\|J\|$ is unbounded by Lemma \ref{lem: mu_change_NSigma}(ii).

{\bf Case 4:} The sequence $\{t_i^{\pm}\}$ ends with some $t^-_m$. The argument in Case 2 implies $\mu_J(t^-_m)>Q_0$. Notice that in this case $\gamma_v[t^-_m,+\infty)$ lies in $\Sigma$. Thus Proposition \ref{prop: highent_highexit}(1) tells us that $\|J\|$ is unbounded.

Hence, for any $v\in S\Sigma^{ext}$, all nonzero perpendicular Jacobi fields along $\gamma_v$ are unbounded. Thus we have finished the proof of Theorem \ref{thm_anosov_ext}.
\end{proof}

\appendix
\section{Estimates on the curvature tensor}

Throughout this section we use notations from Section~\ref{construction 1}.

\subsection{The curvature tensor for the deformation to negative sectional curvature}

For any $\theta_0\in S$, let $\{e_i\}_{i=1}^{n-1}$ be the an orthonormal basis of $g_0$ such that $h(e_i,e_i)=2\lambda_{i}(0,\theta)$. Consider normal coordinates $\{x_i\}_{i=1}^{n-1}$ on $S$ for $g_0$ in a neighborhood of $(0,\theta_0)$ such that $\frac{\partial}{\partial x_i}|_{(0,\theta_0)}=e_i$. For notational convenience we denote by $x_0:=t$. 

\begin{lemma}[The above setting, also see Section~5.1]\label{lem: bdd_curv_tensor_C12}
We use the setting described in this section. Let $\eps>0$. Consider the manifold $[0,1+\eps]\times S$ with Riemannian metric $\tilde g_{\ell,\eps}=dt^2+\tilde g_t$ where
\begin{equation*}
\tilde g_t = \rho(t-\eps)g_0+f_\ell(t)h \quad \text{for all}\quad t\in [0,1+\eps].
\end{equation*}

Then, there exists a constant $C_1=C_1(g, \rho)$ such that for any $i,j,k\in\{1,...,n-1\}$ and $(t_0,\theta_0)\in[0,1+\eps]\times S$,
\begin{equation*}
|R^0_{ijk}(t_0,\theta_0)|<C_1f'_\ell(t_0),
\end{equation*}
where $R^0_{ijk} = \langle R_{\ell,\eps}(\frac{\partial}{\partial x_j},\frac{\partial}{\partial x_k})\frac{\partial}{\partial x_i},\frac{\partial}{\partial x_0}\rangle$ is the coefficient of the Riemann curvature tensor with respect to $\{x_i\}-$coordinates.
\end{lemma}
\begin{proof}
Let $(g_0)_{ij}=g_0\left(\frac{\partial}{\partial x_i},\frac{\partial}{\partial x_j}\right)$ and $h_{ij}=h\left(\frac{\partial}{\partial x_i},\frac{\partial}{\partial x_j}\right)$. Recall that $x_0=t$ and $\{x_i\}_{i=1}^{n-1}$ are normal coordinates near $(0,\theta_0)$  such that $\frac{\partial}{\partial x_i}|_{(0,\theta_0)}=e_i$. We have
\begin{align}\label{eq: normal_coordinates_negative}
&(g_0)_{ij}(0,\theta_0)=\delta_{ij}, \qquad h_{ij}(0,\theta_0) = 2\lambda_i(0,\theta_0)\delta_{ij},\\
&\frac{\partial}{\partial x_k}(g_0)_{ij}(0,\theta_0)=0 \quad \text{and} \quad \nabla^0_{e_i}\frac{\partial}{\partial x_j}=0 \text{ for all } i,j,k\in\{1,\ldots,n-1\},  \nonumber
\end{align}

Moreover, the metric tensor of $\tilde g_{\ell,\eps}$ in coordinates $\{x_0, x_1,\ldots, x_{n-1}\}$ defined in a neighborhood $O_{t_0,\theta_0}$ of $(t_0,\theta_0)$ on $[0,1+\eps]\times S$ has the following entries:
\begin{align}\label{metric_tensor_negative}
&\bar g_{00} = \tilde g_{\ell,\eps}(T, T)=1, \,\,\,\,\,\,\bar g_{0j} = \bar g_{j0}=\tilde g_{\ell,\eps}\left(T,\frac{\partial}{\partial x_j}\right)=0 \text{ for all } j\in\{1,\ldots,n-1\},\\
&\bar g_{ij} = \tilde g_{\ell,\eps}\left(\frac{\partial}{\partial x_i},\frac{\partial}{\partial x_j}\right) = \rho(t-\eps)(g_0)_{ij}+f_\ell(t)h_{ij} \text{ for all } i,j\in\{1,\ldots, n-1\}.\nonumber
\end{align}
Thus, using \eqref{eq: normal_coordinates_negative}, for any $i,j,k\geq 1$, the Christoffel symbols $\Gamma^0_{ij}$ for $\tilde g_{\ell,\eps}$ in $O_{(t_0,\theta_0)}$ and their partial derivatives are
\begin{eqnarray*}
\Gamma^0_{ij}(t,\theta) &=& \frac{1}{2}\bar g^{00}\left(\frac{\partial}{\partial x_j}\bar g_{i0}+\frac{\partial}{\partial x_i}\bar g_{j0}-\frac{\partial}{\partial t}\bar g_{ij}\right) (t,\theta)= -\frac{1}{2}\frac{\partial}{\partial t}\bar g_{ij} (t,\theta)\\
&=&-\frac{1}{2}\rho'(t-\eps)(g_0)_{ij} (0,\theta)-\frac{1}{2}f'_\ell(t)h_{ij} (0,\theta);
\end{eqnarray*}
$$\frac{\partial}{\partial x_k}\Gamma^0_{ij}(t,\theta)=-\frac{1}{2}\rho'(t-\eps)\frac{\partial}{\partial x_k}(g_0)_{ij}(0,\theta)-\frac{1}{2}f'_\ell(t)\frac{\partial}{\partial x_k}h_{ij}(0,\theta). $$
In particular, at $(t_0,\theta_0)$, they are
$$\Gamma^0_{ij}(t_0,\theta_0) = \left(-\frac{1}{2}\rho'(t_0-\eps)-f'_\ell(t_0)\lambda_i(0,\theta_0)\right)\delta_{ij};$$
$$\frac{\partial}{\partial x_k}\Gamma^0_{ij}(t_0,\theta_0)=-\frac{1}{2}f'_\ell(t_0)\frac{\partial}{\partial x_k}h_{ij}(0,\theta_0) = -\frac{1}{2}f'_\ell(t_0)\left(\nabla^0_{e_k} h\right)(e_i,e_j),$$
where $\nabla^0$ is the covariant derivative of tensor at $S$. 

For general $\Gamma^i_{jk}$, by \eqref{metric_tensor_negative}, we have
$$\frac{\partial}{\partial x_i}\bar g_{jk}(t_0,\theta_0)=\rho(t_0-\eps)\frac{\partial}{\partial x_i}(g_0)_{jk}(0,\theta_0)+f_\ell(t_0)\frac{\partial}{\partial x_i}h_{jk} (0,\theta_0)=f_\ell(t_0)(\nabla^0_{e_i}h)(e_j,e_k),$$
Thus, for all $i,j,k\geq 1$,
\begin{align*}
\Gamma^i_{jk}(t_0,\theta_0) &= \frac{1}{2}\bar g^{il}\left(\frac{\partial}{\partial x_k}\bar g_{lj}+\frac{\partial}{\partial x_j}\bar g_{lk}-\frac{\partial}{\partial x_l}\bar g_{jk}\right)(t_0,\theta_0) \\
&= \frac{1}{2}\bar g^{ii}\left(\frac{\partial}{\partial x_k}\bar g_{ij}+\frac{\partial}{\partial x_j}\bar g_{ik}-\frac{\partial}{\partial x_i}\bar g_{jk}\right)(t_0,\theta_0) \\\
&=\frac{f_\ell(t_0)}{2}\frac{(\nabla^0_{e_k}h)(e_i,e_j)+(\nabla^0_{e_j}h)(e_i,e_k)-(\nabla^0_{e_i}h)(e_j,e_k)}{\rho(t_0-\eps)+2f_\ell(t_0)\lambda_i(0,\theta_0)}.
\end{align*}
Let 
$$D_h := \max\limits_{ i,j,k\in\{1,\ldots, n-1\}}\left\{(\nabla^0_{u_k}h)(u_i,u_j) \big|\{u_l\}_{l=1}^{n-1} \text{ is an orthonormal basis of }g_0\right\}.$$
 Then, we have $\left|\frac{\partial}{\partial x_k}\Gamma^0_{ij}(t_0,\theta_0)\right|\leq D_hf'_\ell(t_0)/2$ and
$$|\Gamma^i_{jk}(t_0,\theta_0)|\leq \frac{3f_\ell(t_0)D_h}{2(\rho(t_0-\eps)+2f_\ell(t_0)\lambda_i(0,\theta_0))}\leq \frac{3D_h}{4\lambda_{\min}(S_0)}.$$
Since $f'_\ell(t_0)\geq 1$, we have
\begin{align*}
|R^0_{ijk}(t_0,\theta_0)| &= \left|\frac{\partial}{\partial x_j}\Gamma^0_{ki}-\frac{\partial}{\partial x_k}\Gamma^0_{ji}+\Gamma^l_{ki}\Gamma^0_{jl}-\Gamma^l_{ji}\Gamma^0_{kl}\right|(t_0,\theta_0)\\
&= \left|\frac{\partial}{\partial x_j}\Gamma^0_{ki}-\frac{\partial}{\partial x_k}\Gamma^0_{ji}+\Gamma^j_{ki}\Gamma^0_{jj}-\Gamma^k_{ji}\Gamma^0_{kk}\right|(t_0,\theta_0)\\
&\leq D_hf'_\ell(t_0)+\frac{3D_h}{4\lambda_{\min}(S_0)}(\|\rho\|_{C^1}+2f'_\ell(t_0)\lambda_{\max}(S_0))\\
&\leq \left[D_h+\frac{3D_h}{4\lambda_{\min}(S_0)}\left(\|\rho\|_{C^1}+2\lambda_{\max}(S_0)\right)\right]f'_\ell(t_0)=:C_1f'_\ell(t_0).
\end{align*}
Thus, we finish the proof of Lemma \ref{lem: bdd_curv_tensor_C12}.
\end{proof}

\subsection{The curvature tensor for the ``rounding" deformation}

For any $\theta\in S$, let $\{e_i\}_{i=1}^{n-1}$ be an orthonormal basis of $h$ such that $\hat h(e_i,e_j)=\delta_{ij}\mu_i(\theta)\delta_{ij}$. Let $\mu_{\max}(S) = \max\{\mu_i(\theta)| i\in\{1,\ldots, n-1\}, \theta\in S\}$ and $\mu_{\min}(S) = \min\{\mu_i(\theta)| i\in\{1,\ldots, n-1\}, \theta\in S\}$. Consider normal coordinates $\{x_i\}_{i=1}^{n-1}$ on $S$ for $h$ in a neighborhood of $(0,\theta_0)$ such that $\frac{\partial}{\partial x_i}|_{(0,\theta_0)}=e_i.$ For notational convenience, we again denote $x_0:=t$.

\begin{lemma}\label{lem: bdd_curv_tensor_C3}
We use the setting described in this section. Let $\eps>0$. Consider the product $[0,1+\eps]\times S$ with Riemannian metric $\hat g_{\ell,\eps}=dt^2+\hat g_t$ where 
$$\hat g_t=f_\ell(t+1+\eps)\left(\rho(t)h+(1-\rho(t))\hat h\right),\,\,\,\,\,t\in[0,1+\eps].$$
 Let 
\begin{equation}\label{M_rho and D_h}
M'_{\rho} = \max\limits_{\tau\in\mathbb R}|\rho'(\tau)|
\quad\text{and} \quad D_{\hat h} = \max\left\{(\nabla^h_{u_k}\hat h)(u_i,u_j)\right\},
\end{equation} 
where the maximum in the definition of $D_{\hat h}$ is taken over $i,j,k\in\{1,\ldots, n-1\}$ and an orthonormal basis $\{u_l\}_{l=1}^{n-1}$ of $h$ which also diagonalizes $\hat h$.

Then, for any $i,j,k\in\{1,\ldots,n-1\}$ and $(t_0,\theta_0)\in[0,1+\eps]\times S$,
\begin{equation*}
\left|R^0_{ijk}(t_0,\theta_0)\right|\leq D_{\hat h}\left(1+\frac{2(1+\mu_{\max}(S))}{2\mu_{\min}(S)}\right)\left(f'_\ell(t_0+1+\eps)+f_\ell(t_0+1+\eps)M'_\rho\right).
\end{equation*}
where $R^0_{ijk} = \langle R_{\ell,\eps}\left(\frac{\partial}{\partial j},\frac{\partial}{\partial x_k}\right)\frac{\partial}{\partial x_i},\frac{\partial}{\partial x_0}\rangle$ is the coefficient of the Riemann curvature tensor with respect to $\{x_i\}-$coordinates.
\end{lemma}

\begin{proof}
Let $h_{ij}=h\left(\frac{\partial}{\partial x_i},\frac{\partial}{\partial x_j}\right)$ and ${\hat h}_{ij}={\hat h}\left(\frac{\partial}{\partial x_i},\frac{\partial}{\partial x_j}\right)$. Recall that $x_0=t$ and $\{x_i\}_{i=1}^{n-1}$ is normal coordinate near $(0,\theta_0)$  such that $\frac{\partial}{\partial x_i}|_{(0,\theta_0)}=e_i$, we have
\begin{align}\label{eq: normal_coordinates}
&h_{ij}(0,\theta_0)=\delta_{ij}, \qquad {\hat h}_{ij}(0,\theta_0) = \mu_i(\theta_0)\delta_{ij},\\
&\frac{\partial}{\partial x_k}h_{ij}(0,\theta_0)=0 \quad \text{and} \quad \nabla^h_{e_i}\frac{\partial}{\partial x_j}=0 \text{ for all } i,j,k\in\{1,\ldots,n-1\},  \nonumber
\end{align}

The metric tensor of $\hat g_{\ell,\eps}$ in coordinates $\{t, x_1,\ldots, x_{n-1}\}$ defined in a neighborhood $O_{t_0,\theta_0}$ of $(t_0,\theta_0)$ on $[0,1+\eps]\times S$ has the following entries:
\begin{align*}
&\bar g_{00} = \hat g_{\ell,\eps}\left(\frac{\partial}{\partial t}, \frac{\partial}{\partial t}\right)=1,\\
&\bar g_{0j} = \bar g_{j,0}=\hat g_{\ell,\eps}\left(\frac{\partial}{\partial t},\frac{\partial}{\partial x_j}\right)=0 \text{ for all } j\in\{1,\ldots,n-1\},\\
&\bar g_{ij} = \hat g_{\ell,\eps}\left(\frac{\partial}{\partial x_i},\frac{\partial}{\partial x_j}\right) = f_\ell(t+1+\eps)(\rho(t)h_{ij}+(1-\rho(t))\hat h_{ij}) \text{ for all } i,j\in\{1,\ldots, n-1\}.
\end{align*}

Thus, using \eqref{eq: normal_coordinates}, the Christofell symbols for $\hat g_{\ell,\eps}$ in $O_{(t_0,\theta_0)}$ are
\begin{equation*}
\Gamma^0_{ij}(t,\theta) =-\frac{1}{2}\left(f_\ell(t+1+\eps)\rho(t)\right)'h_{ij}-\frac{1}{2}\left(f_\ell(t+1+\eps)(1-\rho(t))\right)'\hat h_{ij} \text{ so }
\end{equation*}

$$\Gamma^0_{ij}(t_0,\theta_0) = \left(-\frac{1}{2}\left(f_\ell(t+1+\eps)\rho(t)\right)'-\frac{1}{2}\left(f_\ell(t+1+\eps)(1-\rho(t))\right)'\mu_i(\theta_0)\right)\Big|_{t=t_0}\delta_{ij} $$
$$\frac{\partial}{\partial x_k}\Gamma^0_{ij}(t_0,\theta_0) = -\frac{1}{2}\left(f_\ell(t+1+\eps)(1-\rho(t))\right)'|_{t=t_0}\left(\nabla^h_{e_k} \hat h\right)(e_i,e_j), $$
\begin{align*}
\Gamma^i_{jk}(t_0,\theta_0)&=\frac{1}{2}\frac{1-\rho(t_0)}{\rho(t_0)+(1-\rho(t_0))\mu_i(\theta_0)}\left((\nabla^h_{e_k}\hat h)(e_i,e_j)+(\nabla^h_{e_j}\hat h)(e_i,e_k)-(\nabla^h_{e_i}\hat h)(e_j,e_k)\right)
\end{align*}
for all $i,j,k\in\{1,\ldots, n-1\}$.

As a result, the coefficients $R^0_{ijk}(t_0,\theta_0)$ of the Riemann curvature tensor are
\begin{equation*}
R^0_{ijk}(t_0,\theta_0) =\frac{\partial}{\partial x_j}\Gamma^0_{ki}(t_0,\theta_0)-\frac{\partial}{\partial x_k}\Gamma^0_{ji}(t_0,\theta_0)+\Gamma^j_{ki}(t_0,\theta_0)\Gamma^0_{jj}(t_0,\theta_0)-\Gamma^k_{ji}(t_0,\theta_0)\Gamma^0_{kk}.
\end{equation*}
 
 Then, using \eqref{M_rho and D_h}, we have
 $$\left|R^0_{ijk}(t_0,\theta_0)\right|\leq D_{\hat h}\left(1+\frac{2(1+\mu_{\max})}{2\mu_{\min}}\right)\left(f'_\ell(t_0+1+\eps)+f_\ell(t_0+1+\eps)M'_\rho\right).$$
\end{proof}

\section{Sectional curvature for a product manifold}

\begin{lemma}\label{lemma: curvature for product}
Consider the product $(c_1,c_2)\times S$ with Riemannian metric $ds^2 = dt^2+f(t)^2g_S$ where $c_1,c_2\in\mathbb R$, $f(t)>0$ for $t\in(c_1,c_2)$, and $g_S$ is a Riemannian metric on $S$. Let $T = \frac{\partial}{\partial t}$. Then,
\begin{enumerate}
\item \label{eigenvalues for product} The shape operator on $S_t$ is given by $\frac{f'(t)}{f(t)}\textup{Id}$;
\item\label{normal for product} For any nonzero $X\in TS$, the sectional curvature of a plane $\sigma_{X,T}$ is given by
\begin{equation*}
K(\sigma_{X,Y}) = -\frac{f''(t)}{f(t)};
\end{equation*}
\item\label{level for product} For any linearly independent $X,Y\in TS$, the sectional curvature of a plane $\sigma_{X,Y}$ is given by 
\begin{equation*}
K(\sigma_{X,Y}) = \frac{1}{f(t)^2}K^{\text{int}}(g_S, \sigma_{X,Y})-\left(\frac{f'(t)}{f(t)}\right)^2;
\end{equation*}
\item\label{all for product} Let $\sigma$ be a plane which is neither tangent nor orthogonal to $S$ and can be expressed as $\sigma=\sigma_{X+aT, Y}$ for some linearly independent $X,Y\in TS$ and $a>0$. Then, the sectional curvature of $\sigma$ is given by 
\begin{equation*}
K(\sigma)=K(\sigma_{X+aT, Y}) = \frac{1}{1+a^2}K(\sigma_{X,Y})+\frac{a^2}{1+a^2}K(\sigma_{Y,T}).
\end{equation*}
\end{enumerate}
\end{lemma}

Thus, we obtain immediately the following.

\begin{corollary}
Consider the product $(c_1,c_2)\times S$ with Riemannian metric $ds^2 = dt^2+f(t)^2g_S$ where $c_1,c_2\in\mathbb R$, $f(t)>0$ for $t\in(c_1,c_2)$, and $g_S$ is a Riemannian metric on $S$. Then,
\begin{enumerate}
\item $ds^2$ has negative curvature if and only if $f''(t)>0$ and $f'(t)^2>K^{\text{int}}(g_S,\sigma)$ for all $t\in(c_1,c_2)$ and any plane $\sigma$ tangent to $S$;
\item if $ds^2$ has constant negative curvature $-\kappa^2$, then 
\begin{equation*}
f(t)=a_\kappa\sinh(\kappa t)+b_\kappa\cosh(\kappa t) \quad\text{where}\quad t\in(c_1,c_2)
\end{equation*}
for some $a_\kappa,b_\kappa\in \mathbb R$ such that $a_\kappa\tanh(\kappa a_\kappa)>-b_\kappa$.
\end{enumerate}
\end{corollary}

\begin{proof}[Proof of Lemma~\ref{lemma: curvature for product}]
We have that $\II_{S_t} = f'(t)f(t)g_S$ and hence, from definition, the shape operator if given by
$$A(t,\theta) = \frac{f'(t)}{f(t)}\text{Id}.$$  
By \eqref{def: normal sectional}, we obtain Lemma~\ref{lemma: curvature for product}\eqref{normal for product}. Since 
$$K^{\text{int}}(f(t)^2g_S,\sigma_{X,Y}) = \frac{1}{f(t)^2}K^{\text{int}}(g_S,\sigma_{X,Y})$$
 for any linearly independent $X,Y\in TS$, by \eqref{Gauss equation}, we obtain Lemma~\ref{lemma: curvature for product}\eqref{level for product}.

Let $X,Y\in T_{(t_0,\theta_0)}S$. We have that the $t$-coordinate and normal coordinates on $S$ for $g_S$ at $(t_0,\theta_0)$ define coordinates on $(c_1,c_2)\times S$. Using those coordinates and the definition of Riemann curvature coefficients, we can obtain that $\langle R(X,Y)Y,T\rangle=0$. Thus, by \eqref{sectional curvature}, we obtain Lemma~\ref{lemma: curvature for product}\eqref{all for product}.
\end{proof}

\section{$C^1$-gluing for functions of special type}\label{app: combine f with construction 2}

\begin{lemma}\label{lemma: combine f with constant 2}
Let $f_\ell(t) = \frac{e^{\ell t}-1}{\ell}$ and let
$$u_\kappa(t)=\frac{1}{\kappa^2}\left(a\sinh(\kappa t)+b\cosh(\kappa t)\right)^2$$
 where $a,b\in\mathbb R$ are such that $a^2+b^2\neq 0$. For any $\tau>0$ there exists $L=L(\tau,a,b)>0$ such that for all $\ell>L$ there exist $\kappa\in\mathbb R$ and $r>-\tau$ such that $f_\ell(\tau)=u_\kappa(\tau+r)$ and $f'_\ell(\tau) = u'_\kappa(\tau+r)$. Moreover, $-\kappa^2\rightarrow-\infty$ as $\ell\rightarrow\infty$. 
\end{lemma}

\begin{proof}
To prove the lemma we need to solve the following system of equations:
\begin{equation*}
\left\{
\begin{aligned}
&2\frac{e^{\ell \tau}-1}{\ell}\kappa^2+a^2-b^2 = 2ab\sinh(2\kappa(\tau+r))+(a^2+b^2)\cosh(2\kappa(\tau+r)),\\
& e^{\ell\tau}\kappa = (a^2+b^2)\sinh(2\kappa(\tau+r))+2 ab\cosh(2\kappa(\tau+r)).
\end{aligned}
\right.
\end{equation*}
Let $p=2\frac{e^{\ell \tau}-1}{\ell}\kappa^2+a^2-b^2$ and $q=e^{\ell\tau}\kappa$.

Thus, if $a^2=b^2$ then 
\begin{equation*}
\left\{
\begin{aligned}
&\kappa = \frac{ab}{a^2+b^2}\frac{\ell}{1-e^{-\ell\tau}}, \\
& r = -\tau+\frac{1-e^{-\ell\tau}}{\ell}\ln\left(\frac{e^{\ell\tau}-1}{\ell}\cdot\frac{2\kappa^2}{a^2+b^2}\right)>-\tau.
\end{aligned}
\right.
\end{equation*}

Otherwise,
\begin{equation*}
\left\{
\begin{aligned}
&\cosh(2\kappa(\tau+r)) = \frac{(a^2+b^2)p-2abq}{(a^2-b^2)^2}, \\
&\sinh(2\kappa(\tau+r)) = \frac{(a^2+b^2)q-2abp}{(a^2-b^2)^2}.
\end{aligned}
\right.
\end{equation*}

Notice that there exists $L'=L'(\tau, a, b)>0$ such that for all $\ell>L'$ we have $e^{2\ell\tau}-4\frac{e^{\ell\tau}-1}{\ell}(a^2-b^2)>0$. Using the fact that $\cosh(2\kappa(\tau+r))^2-\sinh(2\kappa(\tau+r))^2=1$, we obtain that for all $\ell>L$ there exists a solution
\begin{equation*}
\left\{
\begin{aligned}
\kappa &= \frac{\sqrt{e^{2\ell\tau}-4\frac{e^{\ell\tau}-1}{\ell}(a^2-b^2)}}{2\frac{e^{\ell\tau}-1}{\ell}}\sim\frac{\ell}{2}\rightarrow \infty \quad\text{as}\quad\ell\rightarrow\infty, \\
r &=-\tau+\frac{1}{2\kappa} \sinh^{-1}\left(\frac{(a^2+b^2)q-2abp}{(a^2-b^2)^2}\right)\\&\sim-\tau+\frac{1}{\ell}\sinh^{-1}\left(\frac{1}{2}(a-b)^2e^{\ell\tau}\ell+2ab(\ell-a^2+b^2)\right)>-\tau\quad \quad\text{as}\quad\ell\rightarrow\infty.
\end{aligned}
\right.
\end{equation*}

Thus, there exists $L=L(\tau,a,b)>0$ required by the lemma.
\end{proof}

\section{Constants in the construction of $M^{ext}$} \label{app: const_in_M_ext}

\begin{enumerate}
\item [(C1)] $M_1=Q_0+4$;
\item [(C2)] \label{cond: eps} $\eps<\frac{2}{\lambda_{\min}(\partial\Sigma)}\ln\cosh\frac{\lambda_{\min}(\partial\Sigma)}{4K_g+4(Q_0+2)^2}$, where $K_g$ comes from Proposition~\ref{prop: C11 extension};
\item [(C3)]\label{cond: delta} $\delta<\min\{\delta_0,\frac{\eps}{2},\frac{2}{\lambda_{\min}(\partial\Sigma)}\ln\cosh\frac{\lambda_{\min}(\partial\Sigma)}{8K_0+8(Q_0+1)^2}\}$ where $K_0$ comes from Proposition~\ref{prop: smoothify} and depends on $\eps$ and $\ell$;
\item [(C4)] $R:=\frac{2}{(Q_0+1)^2}+1+\frac{C_0}{Q_0+2}+\frac{2}{Q_0+3}\tanh^{-1}\frac{Q_0+2}{Q_0+3}.$
\end{enumerate}

\bibliography{Riemannian_Anosov_extension_-_2nd_arxiv}{}
\bibliographystyle{alpha}

\Addresses
\end{document}